\documentclass[10pt,reqno]{amsart}
\usepackage{amsmath}
\usepackage{amsthm}
\usepackage{amssymb}
\usepackage{graphicx}
\usepackage{amsfonts}
\usepackage{latexsym}
\usepackage{setspace}
\usepackage{mathdots}
\usepackage{mathrsfs}

\usepackage{bm}
\usepackage{tkz-linknodes}

\usepackage[active]{srcltx}

\numberwithin{equation}{section}
\theoremstyle{plain}
\newtheorem{Proposition}[equation]{Proposition}
\newtheorem{Corollary}[equation]{Corollary}
\newtheorem{Theorem}[equation]{Theorem}
\newtheorem{Lemma}[equation]{Lemma}
\newtheorem{lemma}[equation]{Lemma}
\theoremstyle{definition}

\newtheorem{Remark}[equation]{Remark}
\newtheorem{Question}[equation]{Question}

\allowdisplaybreaks

\def\C{\mathbb{C}}
\def\R{\mathbb{R}}
\def\ov{\overline}

\begin{document}

\bibliographystyle{amsplain}

\title{On a theorem of Livsic}

\author[A.~Aleman]{Alexandru Aleman}
	\address{Department of Mathematics, Lund University, P.O. Box 118, S-221 00 Lund, Sweden}
	\email{aleman@maths.lth.se}
	
\author{R. T. W. Martin}
\address{Department of Mathematics and Applied Mathematics, University of Cape Town, Cape Town, South Africa}
\email{rtwmartin@gmail.com}

\author[W.T.~Ross]{William T. Ross}
	\address{Department of Mathematics and Computer Science, University of Richmond, Richmond, VA 23173, USA}
	\email{wross@richmond.edu}

\begin{abstract}

The theory of symmetric, non-selfadjoint operators has several deep applications to the complex function theory of certain reproducing kernel
Hilbert spaces of analytic functions, as well as to the study of ordinary differential operators such as Schrodinger operators in mathematical physics. Examples of simple symmetric operators include multiplication operators on various spaces of analytic functions such as model subspaces of Hardy spaces,  deBranges-Rovnyak spaces and Herglotz spaces, ordinary differential operators (including Schrodinger operators from quantum mechanics), Toeplitz operators, and infinite Jacobi matrices.

In this paper we develop a general representation theory of simple symmetric operators with equal deficiency indices, and obtain a collection of results which refine and extend classical works of Krein and Livsic.  In particular we provide an alternative proof of a theorem of Livsic which characterizes when two simple symmetric operators with equal deficiency indices are unitarily equivalent, and we provide a new, more easily computable formula for the Livsic characteristic function of a simple symmetric operator with equal deficiency indices.

\end{abstract}

\maketitle

\section{Introduction}

For $n \in \mathbb{N} \cup \{\infty\}$ let $\mathcal{S}_{n}(\mathcal{H})$ denote the set of simple, closed, symmetric, densely defined linear transformations $T: \mathscr{D}(T) \subset \mathcal{H} \to \mathcal{H}$ with deficiency indices $(n, n)$. By this we mean that $T$ is a linear transformation defined on a dense domain $\mathscr{D}(T)$  in a complex separable Hilbert space $\mathcal{H}$ which satisfies the properties:
\begin{equation} \label{symmetric-defn}
\langle T x, y \rangle = \langle x, T y \rangle, \quad \forall x, y \in \mathscr{D}(T), \quad \mbox{($T$ is \emph{symmetric})};
\end{equation}
\begin{equation} \label{simple-defn}
\bigcap_{\Im \lambda \not = 0} \mbox{Rng}(T - \lambda I) = \{0\}, \quad \mbox{($T$ is \emph{simple})};
\end{equation}
\begin{equation}
\mbox{$\{(x, T x): x \in \mathscr{D}(T)\}$ is a closed subset of $\mathcal{H} \oplus \mathcal{H}$}, \quad \mbox{($T$ is \emph{closed})};
\end{equation}
\begin{equation}
\dim \mbox{Rng}(T - i I)^{\perp} = \dim \mbox{Rng}(T + i I)^{\perp}  = n \quad \mbox{(\emph{$T$ has equal deficiency indices})}.
\end{equation}

Condition \eqref{simple-defn} ($T$ is simple) can be restated equivalently as: $T$ is simple if there does not exist a (non-trivial) subspace invariant under $T$ such that the restriction of $T$ to this subspace is self-adjoint \cite{A-G}. We also point out that
$$n = \dim \mbox{Rng}(T - w I)^{\perp} = \dim \mbox{Rng}(T - z I)^{\perp}, \quad \Im w > 0, \Im z < 0,$$
that is to say, the deficiency indices $ \dim \mbox{Rng}(T - w I)$ are constant for $w$ in the upper $\C_{+} := \{\Im z > 0\}$ and lower  $\C_{-} := \{\Im z < 0\}$ half planes \cite[Section 78]{A-G}. As we will discuss later in this paper, examples of operators which satisfy the above properties include certain classes of Sturm-Liouville operators, Schr\"{o}dinger operators, unbounded Toeplitz operators on the Hardy space, and multiplication operators on various spaces of analytic functions on $\C_{+}$ and $\C \setminus \R$.

The purpose of this paper is to rediscover and improve upon a theorem of Livsic \cite{Livsic-1, Livsic-2} (see Theorem \ref{intro-Livsic} below) which characterizes when $T_1 \in \mathcal{S}_n(\mathcal{H}_{1})$ and $T_2 \in \mathcal{S}_{n}(\mathcal{H}_2)$ are unitarily equivalent, written $T_{1} \cong T_2$. Let us review Livsic's theorem when $n < \infty$.
For $T \in \mathcal{S}_n(\mathcal{H})$ we know, since $T$ has equal deficiency indices,  that $T$ has (canonical) self-adjoint extensions $T': \mathscr{D}(T') \subset \mathcal{H} \to \mathcal{H}$.  If $\{u_j\}_{j = 1}^{n}$ is an orthonormal basis for $\mbox{Rng}(T + i I)^{\perp} = \mbox{Ker}(T^* -i I)$,  define the matrix-valued function $w_{T}$ on $\C_{+}$ by
\begin{equation} \label{Livsic-char}
 w_T (z) := b(z) B(z) ^{-1} A(z), \quad z \in \C_{+},
 \end{equation} where
 \begin{equation} \label{Blaschke-b}
 b(z) := \frac{z-i}{z+i},
 \end{equation}
is the single Blaschke factor defined on $\C_{+}$ with zero at $z = i$,
$$B(z) :=  \left[ \langle \left( I + (z - i) (T' - z I)^{-1} \right) u_j, u_k\rangle \right]_{1 \leqslant j, k \leqslant n}, $$ and
$$  A (z) := \left[  \langle \left( I + (z + i) (T' - z)^{-1} \right) u_j, u_k \rangle \right]_{1 \leqslant j, k \leqslant n}. $$

The function $w_T$ in \eqref{Livsic-char}, called the \emph{Livsic characteristic function} for $T$, is a contractive matrix-valued analytic function on $\C_{+}$.  Moreover, given any contractive, matrix-valued analytic function $w$ on $\C_{+}$ with $w(i) = 0$ there is a closed, simple, symmetric, linear transformation $T$ with $w = w_{T}$ (however this symmetric linear transformation $T$ is not necessarily densely defined). The characteristic function $w_T$ of $T$ is essentially independent of the choice of self-adjoint extension $T'$ and the choice of orthonormal basis $\{ u_j \}_{j = 1}^{n}$, i.e.,  if $T' _k, k = 1, 2, $ are two self-adjoint extensions of $T$, $\{ u _j ^{(k)} \}_{j = 1}^{n}, k = 1, 2,$ are orthonormal bases of $\mbox{Ker}(T^* -i I)$, and $w_k, k = 1, 2,$ are the characteristic functions of $T$ constructed using the $T' _k$ and $\{ u_j ^{(k)} \}_{j = 1}^{n}$, then there exists two constant unitary matrices $Q$ and $R$ such that
\begin{equation} \label{RQ-equiv}
w_1(z) = R w_2(z) Q, \quad z \in \C_{+}.
\end{equation}
For this reason we say that two characteristic functions $w_1, w_2$ are \emph{equivalent} if condition \eqref{RQ-equiv} for some constant unitary matrices $Q$ and $R$.
Livsic's theorem is the following:

\begin{Theorem}[Livsic \cite{Livsic-1, Livsic-2}] \label{intro-Livsic}
The operators $T_1 \in \mathcal{S}_n(\mathcal{H}_1)$ and $T_2 \in \mathcal{S}_n(\mathcal{H}_2)$ are unitarily equivalent if and only if $w_{T_1}$ and $w_{T_2}$ are equivalent.
\end{Theorem}

Lisvic's original proof of this result uses the spectral theorem for self-adjoint operators \cite{A-G}. Here is a brief sketch: One direction of the proof is straightforward.  Indeed if $T_1$ and $T_2$ are unitarily equivalent, one can find self-adjoint extensions $T_i ', i = 1, 2,$ of the $T_i$ which are unitarily equivalent, and, if one uses these extensions to construct the characteristic functions $w_{T_i}$ as in \eqref{Livsic-char}, it follows that these characteristic functions will be equivalent in the sense of \eqref{RQ-equiv}. To prove the
converse, if $w_{T_1}$ and $w_{T_2}$ are equivalent, one can, without loss of generality, assume they are equal (they can be made equal by choosing the orthonormal bases $\{ u _j ^{(1)} \}_{j = 1}^{n}, \{ u _j ^{(2)} \}_{j = 1}^{n}$ for $\mbox{Ker}(T_1 ^* -i I )$ and $\mbox{Ker}(T_{2}^{*} - i I)$ respectively, and
the self-adjoint extensions $T_{1}^{'}, T_{2}^{'}$ used to construct the characteristic functions $w_{T_1 }, w_{T_2}$ appropriately). A calculation using \eqref{Livsic-char} shows that $w_{T_1} = w _{T_2}$ implies that
$\Omega _1 =  \Omega _2$ where
\begin{equation}  \label{Her-rep-intro}
\Omega _k (z)  := \frac{ I + w_{T_k} (z) }{ I - w_{T_k} (z)} = \frac{1}{i\pi} \int _{-\infty} ^\infty \frac{1}{t-z} \Lambda_k (dt), \end{equation}
and
 $\Lambda_p$ are $n\times n$ unital positive matrix-valued measures such that
 $$\Lambda_p (\Delta ) = 4\pi ^2 (1+t^2) \left[ \langle \chi _\Delta (T' _p ) u_j ^{(p)}  , u _k ^{(p)} \rangle \right]_{1 \leqslant j, k \leqslant n}, \quad p = 1, 2,$$ for any Borel subset $\Delta \subset \mathbb{R}$. Here $\chi _\Delta $ denotes the characteristic function of the Borel set $\Delta$, and $\chi _\Delta (T ' _p )$ defines a unital projection-valued measure using the functional calculus for self-adjoint operators. The uniqueness of the Herglotz representation in \eqref{Her-rep-intro}, along with the fact that $\Omega _1 = \Omega _2$, implies that $\Lambda_1 = \Lambda_2$. Since the $\{ u _j ^{(p)} \}_{j = 1}^{n}, p = 1, 2, $ are generating bases
for the unitary operators $b(T'_{p})$
 (this follows from the simplicity of the $T_p$), it follows that the $T'_{p}, p = 1, 2,$ and hence the $T_p$ are unitarily equivalent.

In this paper we give an alternate proof of Livsic's theorem (Theorem \ref{intro-Livsic}) using reproducing kernel Hilbert spaces of analytic functions. In particular, in Theorem \ref{T-main-vector-kernel} below, we factor these reproducing kernels in a particular way which yields the Livsic characteristic function. By doing this we accomplish several things. First, our factorization of reproducing kernels technique gives us further insight into what makes Livsic's theorem work and lets us see the characteristic function in a broader context. Second, our alternate proof is more abstract and thus gives us more latitude in computing the characteristic function since computing $w_T$, as it is defined by \eqref{Livsic-char}, involves a self adjoint extension of $T$, which can be difficult to compute, as well as a resolvent, which is also difficult to compute. Third, by associating, in certain circumstances, $T \in \mathcal{S}_{n}(\mathcal{H})$ with multiplication by the independent variable on a deBranges-Rovnyak space, we can gain further information about  some function theory properties of the associated Livsic function.
Fourth, our proof handles the $n = \infty$ case for which $w_T$ becomes an contractive operator-valued analytic function on $\C_{+}$.

The main results of this paper will be to (i) associate any $T \in \mathcal{S}_{n}(\mathcal{H})$ with a vector-valued reproducing kernel Hilbert space of analytic functions on $\C \setminus \R$ (Propositions \ref{H-Gamma} and \ref{H-Gamma2}); (ii) associate the kernel function for this space with the Livsic characteristic function (Theorem \ref{T-main-vector-kernel}); (iii) compute the Livsic function for the operators of differentiation and double differentiation, Sturm-Liouville operators, unbounded symmetric Toeplitz operators, and symmetric operators which act as multiplication by the independent variable in Lebesgue spaces, Herglotz spaces, and deBranges-Rovnyak spaces; (iv) show that $T$ is unitarily equivalent to multiplication by the independent variable on a Herglotz space (Theorem \ref{Mz-Herglotz}); (v) show that when $n < \infty$ and the Livsic characteristic function for $T$ is an extreme point of the unit ball of the $n\times n$ matrix-valued bounded analytic functions on $\C_{+}$, then $T$ is unitarily equivalent to multiplication by the independent variable on an associated vector-valued deBranges-Rovnyak space (Corollary \ref{Mz-DR}); (vi) use this equivalence to show, when the Livsic function $V$ for $T$ is an extreme point, that the angular derivative of $(V \circ b^{-1}) \vec{k}$ at $z = 1$ does not exist for any $\vec{k} \in \C^{n}$ (Corollary \ref{Mz-DR}).

Finally we mention that some of the results we prove here, like Livsic's theorem and the fact that every $T \in \mathcal{S}_{n}(\mathcal{H})$ can be realized as multiplication by the independent variable on some Lebesgue space,  are known (and we will certainly point out the original sources) but the main wrinkle here is that they can be obtained via reproducing kernel Hilbert spaces and factorization of kernel functions for these spaces. Moreover, via deBranges-Rovnkay spaces, we gain some additional information about the Livsic function. As demonstrated above with the the sketch of the proof of Livsic's theorem, the original proofs used the spectral theorem, which certainly adds efficiency and utility (and even elegance) but not computability.

We would be remiss if we did not point out a paper of Poltoratski and Makarov \cite{MR2215727} which uses a different model than ours to associate operators with inner functions and classical model spaces of the upper-half plane. In particular they use these results to solve specific problems associated with Schr\"{o}dinger operators.

\section{A model operator}

The main idea, going back to Krein \cite{MR1466698, MR0011170, MR0011533, MR0048704}, and used many times before \cite{MR1784638, Langer, Martin1}, in examining symmetric operators is the idea of a vector-valued reproducing kernel Hilbert space of analytic functions associated with a symmetric operator. For $T \in \mathcal{S}_n(\mathcal{H})$, $n \in \mathbb{N} \cup \{\infty\}$, let $\mathcal{K}$ be any complex separable Hilbert space whose dimension is
$$n = \mbox{Rng}(T+ i I)^{\perp} = \mbox{Rng}(T - i  I)^{\perp}.$$
 When $n \in \mathbb{N}$, one usually takes $\mathcal{K}$ to be $\C^{n}$, with the standard inner product
 $$\langle \vec{z}, \vec{w}\rangle_{\C^n} := \sum_{j = 1}^{n} z_{j} \overline{w_j}.$$

 \subsection{The model}
 If $\mathcal{B}(\mathcal{K}, \mathcal{H})$ is the space of bounded linear operators from $\mathcal{K}$ to $\mathcal{H}$, we say that $\Gamma: \C \setminus \R \to \mathcal{B}(\mathcal{K}, \mathcal{H})$ is a
  \emph{model} for $T$ if $\Gamma$ satisfies the following conditions:
\begin{equation} \label{I}
\Gamma: \C \setminus \R \to \mathcal{B}(\mathcal{K}, \mathcal{H}) \quad \mbox{is co-analytic};
\end{equation}
\begin{equation}
\Gamma(\lambda): \mathcal{K} \to \mbox{Rng}(T  - \lambda I)^{\perp} \quad \mbox{is invertible for each $\lambda \in \C \setminus \R$};
\end{equation}
\begin{equation}
\Gamma(z)^{*} \Gamma(\lambda): \mathcal{K} \to \mathcal{K} \quad \mbox{is invertible for all $\lambda, z \in \C_{+}$ or $\lambda, z \in \C_{-}$}; \label{equation:condition}
\end{equation}
\begin{equation}
\bigvee_{\Im \lambda \not = 0} \mbox{Rng} \Gamma(\lambda) = \mathcal{H},
\end{equation}
where $\bigvee$ denotes the closed linear span.

\begin{Proposition}
Every $T \in \mathcal{S}_{n}(\mathcal{H})$ has a model. \label{prop:model}
\end{Proposition}

The proof of this proposition needs a little set up. Given a closed densely-defined operator $T$ with domain $\mathscr{D} (T) \subset \mathcal{H}$, a point $z \in \mathbb{C}$ is called a \emph{regular point} of $T$ if $T -z I$ is bounded below on $\mathscr{D} (T)$, i.e., $\|(T - z I) x\| \geqslant c_{z} \|x\|$ for all $x \in \mathscr{D}(T)$. Let $\Omega $ denote the set of all regular points of $T$. If $T \in \mathcal{S} _n (\mathcal{H})$, then since $T$ is symmetric  we have $\mathbb{C} \setminus \mathbb{R} \subset \Omega \subset \mathbb{C}$. $T$ is called \emph{regular} if $\Omega = \mathbb{C}$. For any $w \in \Omega$, let
$$\mathscr{D} _w := \mathscr{D} (T) + \mbox{Ker}(T^* -w I ),$$
and define $T_w := \overline{T^* | \mathscr{D} _w}$, the closure of $T^*|\mathscr{D}_{w}$.
\begin{lemma}
    Suppose that $w \in \Omega$ is a regular point of $T \in \mathcal{S} _n (\mathcal{H})$. The spectrum of $T_w$ is contained in $\overline{\mathbb{C} _{+}} $, $\mathbb{R}$ or $\overline{\mathbb{C} _{-}}$ when $w \in \mathbb{C} _+ , \ \mathbb{R} \cap \Omega $ or $\mathbb{C} _-$ respectively.
\end{lemma}

\begin{proof}

    We will prove the lemma when $w \in \mathbb{C} _+$. The proofs of the other two cases are analogous. Note that when $w \in \mathbb{R} \cap \Omega$  a proof that $T_w$ is in fact self-adjoint is found in \cite[Section 83]{A-G}.

    First we show, for any $u \in \mathscr{D} _w$, that $\Im \left( \langle T_w u , u \rangle  \right) \geqslant 0$ (assuming $w \in \mathbb{C} _+$). Since $\mathscr{D} _w$ is by definition a core for $T_w$,
this will prove that $T_w -z I $ is bounded below for any $z \in \mathbb{C} _-$. Any $u \in \mathscr{D} _w$ can be written as $u = v + \psi _w $ where $u \in \mathscr{D}(T)$ and $\psi _w \in \mbox{Ker} (T^* -w I)$.
It follows that
\begin{eqnarray}
\langle T_w u , u \rangle & = & \langle T v + w \psi _w , v + \psi _w \rangle \nonumber \\
    & = & \langle T v , v \rangle + (\overline{w} \langle v , \psi _w \rangle + w \langle \psi _w , v \rangle ) + w \| \psi _w \| ^2 ,
\end{eqnarray}
and so \[ \Im  \langle T_w u , u \rangle  = \Im(w) \| \psi _w \| ^2 \geqslant 0. \] To show that the spectrum of $T_w$ is contained in $\overline{\C_{+}}$, it remains
to verify that $T_w - z I $ is onto for any $z \in \mathbb{C} _-$.
First we show that $T_w - \overline{w} I$ is onto. If $\phi \in \mathcal{H}$ and $\phi \perp \mbox{Rng} (T_w - \overline{w} I )$ then
$\phi \perp \mbox{Rng} (T - \overline{w} I)$, so that $\phi \in \mbox{Ker}(T^* -w I)$, and \[0 = \langle (T_w -\overline{w} I ) \phi , \phi \rangle = (w -\overline{w} ) \| \phi \| ^2. \]
Since $w \notin \mathbb{R}$ this shows that $\phi =0$ and proves that $T_w - \overline{w} I$ is onto. The fact that $\Im \langle T_w u , u \rangle  \geqslant 0$ for any $u \in \mathscr{D} _w$
implies that every $z \in \mathbb{C} _-$ is regular for $T_w$.  By \cite[Section 78]{A-G} the dimension of $\mbox{Rng} ( T_w - z I ) ^\perp $ is constant for $z$ in any connected component of the set of regular
 points of $T_w$. It follows that $T_w -z I$ is onto for all $z \in \mathbb{C} _-$.
\end{proof}

\begin{proof}[Proof of Proposition \ref{prop:model}]
 This proof is adapted from a resolvent formula from \cite[Sec.~84]{A-G}.  Let $T'$ be any (canonical) self-adjoint extension of $T$ and note that for each fixed  $\lambda \not \in \R$, the operator $T' - \lambda I: \mathscr{D}(T') \to \mathcal{H}$ is onto. Define
 \begin{equation} \label{K-form}
U_{\lambda} := (T' - i I) (T' - \lambda I)^{-1} = I + (\lambda - i) (T' - \lambda I)^{-1}
\end{equation} and observe that $U_{\lambda} : \mathcal{H} \to \mathcal{H}$ is one-to-one and onto. Moreover, for any $x \in \mathcal{H}$, the map $\lambda \mapsto U_{\lambda} x$ is an $\mathcal{H}$-valued analytic function on $\C \setminus \R$. By \cite[Sec.~84]{A-G}  $U_{\lambda}$ is a bounded invertible operator from $\mbox{Rng}(T + iI)^{\perp}$ onto $\mbox{Rng}(T - \overline{\lambda} I)^{\perp}$. Indeed, the inverse of $U_{\lambda}$ is
$$(T' - \lambda) (T' - i I)^{-1}.$$

Recall that $\mathcal{K}$ is any complex separable Hilbert space whose dimension is equal to
$$\dim \mbox{Rng}(T - i I)^{\perp} = \dim \mbox{Rng}(T + i I)^{\perp}  = n.$$ Let $j$ be any bounded isomorphism from $\mathcal{K}$ onto $\mbox{Rng}(T + i I)^{\perp}$. Finally define
\begin{equation} \label{Krein-trick}
\Gamma(\lambda) := U_{\overline{\lambda}} j.
\end{equation}
 With the exception of condition (\ref{equation:condition}),
 one can easily check that $\Gamma$ satisfies the conditions of a model for $T$.
We will now show that $\Gamma( z ) ^* \Gamma (\lambda)$ is invertible whenever $z,\lambda $ are both in $\mathbb{C} _+$ or both in $\mathbb{C} _-$. In particular, this will prove that $\Gamma$ obeys condition (\ref{equation:condition}).

Without loss of generality assume that $\mathcal{K} := \mathbb{C} ^n$ with canonical orthonormal basis $\{ e_k \}_{k = 1}^{n}$. When $n=\infty$, we define $\mathbb{C} ^\infty := \ell^2(\mathbb{N})$, the Hilbert space of square-summable sequences of complex numbers. Since $j$ is a bounded isomorphism, we see that the set $\{ \gamma _k (i) \} _{k=1} ^n$ where $\gamma _k (i) := j e_k$ is a basis (in fact a Riesz basis) for $\mbox{Ker} (T^* -i I)$. For any $\lambda \in \mathbb{C} _+$, the set $\{\gamma _k (\lambda)\}_{k = 1}^{n}$ where $\gamma _k (\lambda) := U_\lambda \gamma _k (i)$ is a
basis (in general non-orthonormal) for $\mbox{Ker}(T^* -\lambda I)$. It follows that we have the matrix representation
$$ \Gamma (\overline{z} ) ^* \Gamma ( \overline{\lambda} ) = \left[ \langle \gamma _j (\lambda ) , \gamma _k (z) \rangle \right]_{1 \leqslant j, k \leqslant n} \label{matrixrep} $$
for $z, \lambda \in \mathbb{C} \setminus \mathbb{R}$. We will need the fact that \{$\gamma _j (z )\}_{j = 1}^{n}$ is actually a Riesz basis for $\mbox{Ker}(T^* - z I )$, \emph{i.e.}, the image of an orthonormal
basis under a bounded, invertible operator. This implies there are constants $0 < c \leqslant C$ such that for any $\psi \in \mbox{Ker}(T^* -z I ) $,
\begin{equation} c \| \psi \| ^2 \leqslant \sum_{j = 1}^{n} | \langle \gamma _j (z) , \psi \rangle | ^2 \leqslant C \| \psi \| ^2 . \label{Riesz} \end{equation}  To see this, choose an orthonormal basis $\{ \delta _j (z ) \}_{j = 1}^{n}$ for $\mbox{Ker}(T^* - z I )$
and let $U : \mbox{Ker}(T^* - z I) \rightarrow \mathbb{C} ^n$ be the isometry defined by $U \delta _j (z) = e_j $. It follows that the linear map
$$V : \mbox{Ker}(T^* - z I ) \rightarrow \mbox{Ker}(T^* - z I), \quad V := U_z j U $$ is invertible. Hence for any $\psi \in \mbox{Ker}(T^* -z I) $,
$$ \sum_{j = 1}^{n} | \langle \gamma _j (z) , \psi \rangle | ^2 = \sum_{j = 1}^{n} | \langle \delta _j (z) , V^* \psi \rangle | ^2 = \| V^* \psi \| ^2, $$ and since $V^*$ is bounded above and below (because $V$ is invertible),
equation (\ref{Riesz}) follows.

Observe that if $\vec{c} \in \mathbb{C} ^n$ we have
 \[ \Gamma (\overline{z}) ^* \Gamma (\overline{\lambda}) \vec{c}  = ( \langle \gamma _k (\lambda ) , \psi _{\vec{c}} (z) \rangle ) _{k=1} ^n,  \]  where
\[ \psi _{\vec{c}} (z)  =  \sum _{k=1} ^n \overline{ c_k } \gamma _k (z) = \Gamma (\overline{\lambda} )  \vec{c} _{c} . \] Here $\vec{c} _c$ denotes the component-wise complex-conjugate of the vector $\vec{c}$.
If $\vec{c}$ has unit norm, then since $\Gamma (\lambda ) : \mathbb{C} ^n \rightarrow \mbox{Ker}(T^* -\overline{\lambda} I )$ is bounded and invertible, it follows that there are constants $c(z), C(z) > 0$ such that
$$c(z) \leqslant \| \psi _{\vec{c}} (z) \| \leqslant C(z), \quad \forall \vec{c} \in \C^{n}, \|\vec{c}\| = 1.$$ Hence in order to prove that $\Gamma (\overline{z} ) ^* \Gamma (\overline{\lambda})$ is bounded below, it suffices to show, for any unit norm $\psi ( z) \in \mbox{Ker}(T^* - z I)$, that the sequence $$\vec{b}  := ( \langle \gamma _k (\lambda) , \psi (z) \rangle )_{k = 1}^{n} $$ is bounded below in the norm of $\mathbb{C} ^n$.

Now suppose that both $z$ and $\lambda$ belong to $\mathbb{C} _+$ or both belong to $\mathbb{C} _-$ and assume that $\Gamma(\overline{z}) ^* \Gamma (\overline{\lambda} )$ is not bounded below. Then, by the discussion above, there exists a sequence of unit norm vectors $\psi _k (z) \in \mbox{Ker}(T^* -z I)$,  such that \[  \vec{b} _k := ( \langle \gamma _j (\lambda) , \psi _k (z) \rangle ) _{j=1} ^n   \rightarrow 0 \] in $\mathbb{C} ^{n}$ norm as $k \to \infty$.
Let $\phi _k = P \psi _k (z)$ be the projection of $\psi _k (z) $ onto $\mbox{Rng} (T - \overline{\lambda} I) = \mbox{Ker}(T^* -\lambda I)^\perp$, where $I-P$ is the projection
onto $\mbox{Ker}(T^* -\lambda I)$. Since we assume that the $\psi _k (z)$
all have unit norm, we see that $\| \phi _k  \| = \| P \psi _k (z) \| \leqslant 1$ for all $k$. Since $\{ \gamma _j (\lambda) \}_{j = 1}^{n}$ is actually a Riesz basis for $\mbox{Ker}(T^* -\lambda I)$, it can be shown
that $\psi _k (z) - \phi_k  \rightarrow 0$ in norm as $k \rightarrow \infty$ so that $\| \phi _k  \|  \rightarrow 1$ . To see this last fact note that by equation (\ref{Riesz}) there is a constant $c >0$ such that
\begin{eqnarray}
 \|  \psi _k (z)  - \phi _k  \| ^2 & =  & \| (I-P ) \psi _k (z) \| ^2\\
 & \leqslant & \frac{1}{c} \sum _j | \langle \gamma _j (\lambda) , (I-P ) \psi _k (z) \rangle | ^2 \nonumber \\
 & = & \frac{1}{c} \sum _j  | \langle \gamma _j (\lambda) ,  \psi _k (z) \rangle | ^2 \nonumber\\
 & = & \frac{1}{c} \| \vec{b} _k \| ^2 _{\mathbb{C} ^n}, \end{eqnarray}
which vanishes as $k \rightarrow \infty$ by assumption.

Since $\phi _k \in \mbox{Rng} (T- \overline{\lambda} I ) $, it follows that $\phi _k = (T- \overline{\lambda} I ) \varphi _k$ for some sequence
$\varphi _k \in \mathscr{D} (T)$.
Now consider
\begin{eqnarray} (T-z I) \varphi _k & = &( T - \overline{\lambda} I) \varphi _k - (z -\overline{\lambda}) \varphi _k \nonumber \\
& = & \phi _k -  (z -\overline{\lambda}) \varphi _k. \label{eq:bndbelow} \end{eqnarray}
We want to show that this vanishes as $k \rightarrow \infty$ and that $\| \varphi _k \| $ is uniformly bounded below in norm. This will show that $T-z I$ is not bounded below, contradicting the fact that $z \in \mathbb{C} _\pm $ is a regular point for $T$.
Since $\overline{\lambda} \in \mathbb{C} _\mp$ is not in the spectrum of $T_z$, and each eigenvector $\psi _k (z)$ is an eigenvector of $T_z$ corresponding to the eigenvalue $z$, we can write
\begin{eqnarray}
(z -\overline{\lambda } ) ^{-1} \phi _k -\varphi _k & = & \left( (z -\overline{\lambda } ) ^{-1} - (T_z -\overline{\lambda}I ) ^{-1} \right) \phi _k  \nonumber \\
& =& \left( (z -\overline{\lambda } ) ^{-1} - (T_z -\overline{\lambda} I) ^{-1} \right) (\phi _k - \psi _k (z) ).
\end{eqnarray}
This vanishes as $k \to \infty$ since $\psi _k (z) - \phi _k \rightarrow 0$ in norm. Moreover, by equation (\ref{eq:bndbelow}),
$$ | z - \overline{\lambda} | \| \varphi _k \| \geqslant \| \phi _k \| - \| (T-z I) \varphi _k \| \rightarrow 1, $$ which shows that the $\| \varphi _k \| $ are uniformly bounded below in norm.
This implies $T-z I$ is not bounded below, contradicting the assumption that $z \in \mathbb{C} \setminus \mathbb{R}$.

The above proves that $\Gamma(z)^* \Gamma (\lambda )$ is bounded below whenever $z, \lambda \in \mathbb{C} _\pm$. Since the adjoint of $\Gamma (z) ^* \Gamma (\lambda)$
is $\Gamma (\lambda )^* \Gamma (z)$, and $\Gamma (z) ^* \Gamma (\lambda)$ is not onto if and only if its adjoint has non-zero kernel, this actually proves that $\Gamma (z) ^* \Gamma (\lambda )$ is invertible
whenever $z, \lambda \in \mathbb{C} _+$ or $z,\lambda \in \mathbb{C} _-$.
\end{proof}

\begin{Remark}
The key point of this, perhaps overly formal, approach is that one is free to choose the model and is not restricted to the one given by the above Krein construction. We will give many examples, and take advantage, of this freedom below.
\end{Remark}

\subsection{The model space}
For a model $\Gamma$ we now define an associated vector-valued Hilbert space of analytic functions $\mathcal{H}(\Gamma)$ associated with our underlying Hilbert space $\mathcal{H}$ on which $T$ acts. For $f \in \mathcal{H}$ define
$$\widehat{f}: \C \setminus \R \to \mathcal{K}, \quad \widehat{f}(\lambda) := \Gamma(\lambda)^{*} f,$$ and
$$\mathcal{H}(\Gamma) := \left\{\widehat{f}: f \in \mathcal{H}\right\}.$$

\begin{Proposition} \label{H-Gamma}
With an inner product on $\mathcal{H}(\Gamma)$ defined by
$$\langle \widehat{f}, \widehat{g} \rangle_{\mathcal{H}(\Gamma)} := \langle f, g \rangle_{\mathcal{H}},$$
$\mathcal{H}(\Gamma)$ is a vector-valued reproducing kernel Hilbert space of analytic functions on $\C \setminus \R$. Moreover, the reproducing kernel function for $\mathcal{H}(\Gamma)$ is
$$K_{\lambda}(z) = \Gamma(z)^{*} \Gamma(\lambda),$$
i.e., for any $a \in \mathcal{K}$ and $\widehat{f} \in \mathcal{H}(\Gamma)$,
\begin{equation} \label{RKHS}
\langle \widehat{f} (\lambda ) , a \rangle_{\mathcal{K}} = \langle \widehat{f}, K_{\lambda}(\cdot) a \rangle_{\mathcal{H}(\Gamma)}.
\end{equation}
\end{Proposition}

\begin{proof}
The only significant things to check here are that (i) $\|\widehat{f}\|_{\mathcal{H}(\Gamma)} = 0$ if and only if $\widehat{f}(\lambda) = 0$ for all $\lambda \in \C \setminus \R$; and (ii) the reproducing kernel formula from \eqref{RKHS}.

Fact (i) follows from the fact that $T$ is simple.
Indeed, for each $\lambda \in \C \setminus \R$,
$$\widehat{f}(\lambda) := \Gamma(\lambda)^{*} f = 0_{\mathcal{K}} \Leftrightarrow f \in \mbox{Ker} (\Gamma(\lambda)^{*} ) = \mbox{Rng}(T - \lambda I).$$
Thus if $\widehat{f}(\lambda) = 0$ for all $\lambda \in \C \setminus \R$, we use  \eqref{simple-defn} to see that $f = 0_{\mathcal{H}}$.

To prove (ii) let $\widehat{f} \in \mathcal{H}(\Gamma)$ and $a \in \mathcal{K}$. Then
\begin{align*}
\langle \widehat{f}, K_{\lambda}(\cdot) a\rangle_{\mathcal{H}(\Gamma)} & = \langle \Gamma(\cdot)^{*} f, \Gamma(\cdot)^{*} \Gamma(\lambda) a \rangle_{\mathcal{H}(\Gamma)}\\
& = \langle f, \Gamma(\lambda) a \rangle_{\mathcal{H}}\\
& = \langle \Gamma(\lambda)^{*} f, a \rangle_{\mathcal{K}}\\
& = \langle \widehat{f}(\lambda), a \rangle_{\mathcal{K}},
\end{align*}
which proves the reproducing kernel formula in \eqref{RKHS}.
\end{proof}

\begin{Proposition} \label{H-Gamma2}
For $T \in \mathcal{S}_{n}(\mathcal{H})$ with model $\Gamma$, the operator $M_{\Gamma}$ on defined on
$$\mathscr{D}(M_{\Gamma}) := \left\{\widehat{f} \in \mathcal{H}(\Gamma): z \widehat{f} \in \mathcal{H}(\Gamma)\right\}$$
by $M_{\Gamma} \widehat{f} = z \widehat{f}$
is densely defined and belongs to $\mathcal{S}_{n}(\mathcal{H}(\Gamma))$. Moreover, $M_{\Gamma}$ is unitarily equivalent to $T$.
\end{Proposition}

\begin{proof}
Let
$$U: \mathcal{H} \to \mathcal{H}(\Gamma), \quad (U f)(z) := \widehat{f}(z) = \Gamma(z)^{*} f$$ and note by Proposition \ref{H-Gamma} that $U$ is an isometric isomorphism. We need to show that if
$$\mathscr{D}(M_{\Gamma}) := \left\{\widehat{f} \in \mathcal{H}(\Gamma): z \widehat{f} \in \mathcal{H}(\Gamma)\right\},$$ then
\begin{equation} \label{UTdomains}
U \mathscr{D}(T) = \mathscr{D}(M_{\Gamma})
\end{equation}
and
\begin{equation} \label{intertwine}
U T = M_{\Gamma} U.
\end{equation}

Let us first show the $\supset$ containment in \eqref{UTdomains}. Indeed let $U f \in \mathscr{D}(M_{\Gamma})$, i.e., $z U f \in \mathcal{H}(\Gamma)$. Then for any $\lambda \in \C \setminus \R$ the function
$$z \mapsto (z - \lambda) U f$$ is a function in $\mathcal{H}(\Gamma)$ which is zero at $\lambda$ and so, using the fact from the proof of the previous proposition that
\begin{equation} \label{Gamma-star-0}
\widehat{f}(z) = \Gamma(z)^{\ast} f = 0 \Leftrightarrow f \in \mbox{Rng}(T - z I),
\end{equation}
we see that
$$(z - \lambda) U f = U f_{\lambda}$$ for some $f_{\lambda} \in \mbox{Rng}(T - \lambda I)$. Thus $
$$f_{\lambda} = (T - \lambda I) g_{\lambda}$, where $g_{\lambda} \in \mathscr{D}(T)$. We now need to prove that
$f = g_{\lambda}$.  Indeed,
\begin{align*}
(z - \lambda) \Gamma(z)^{*} f  & = (z - \lambda) (U f)(z)\\
& = (U f_{\lambda})(z)\\
& = (U (T - \lambda I) g_{\lambda})(z)\\
& = \Gamma(z)^{*} (T - \lambda I) g_{\lambda}\\
& = \Gamma(z)^{*} ((T - z I) g_{\lambda} + (z - \lambda) g_{\lambda})\\
& = \Gamma(z)^{*} (T - z I) g_{\lambda} + (z - \lambda)\Gamma(z)^{*} g_{\lambda}\\
& = 0 + (z - \lambda) \Gamma(z)^{*} g_{\lambda}.
\end{align*}
The last equality follows from \eqref{Gamma-star-0}.
This means that $\Gamma(z)^{*} f = \Gamma(z)^{*} g_{\lambda}$ for all $z \in \C \setminus \R$. From \eqref{Gamma-star-0} and the fact that $T$ is simple (see \eqref{simple-defn}) it follows that $f  = g_{\lambda}$. Thus we have shown the $\supset$ containment in \eqref{UTdomains}.

For the $\subset$ containment in \eqref{UTdomains}, let $f \in \mathscr{D}(T)$. Then for any $z \in \C \setminus \R$,
\begin{align*}
(U T f)(z) & = \Gamma(z)^{*} T f\\
& = \Gamma(z)^{*} ((T - z I) f + z f)\\
& = 0 + z \Gamma(z)^{*} f\\
& = M_{\Gamma} (U f)(z).
\end{align*}
This proves the $\subset$ containment in \eqref{UTdomains} along with the intertwining identity in \eqref{intertwine}.
\end{proof}

\begin{Remark} As mentioned earlier, we are not constrained by the Krein trick \eqref{Krein-trick} in selecting our model $\Gamma$ for $T$. There are other methods of constructing a model. For example,  when $n < \infty$, we can use Grauert's theorem, as was used to prove a related result for bounded operators in \cite{Cow-Doug}, to find an analytic vector-valued function
$$\gamma(\lambda) := (\gamma(\lambda)_1, \cdots, \gamma(\lambda)_{n}),$$
where $\{\gamma(\lambda)_1, \cdots, \gamma(\lambda)_{n}\}$ is a basis for $\mbox{Rng}(T - \overline{\lambda} I)^{\perp}$. Then, if $\{e_j\}_{j = 1}^{n}$ is the standard basis for $\C^n$,  we can define our model for $T$ to be
\begin{equation} \label{model-anything}
\Gamma(\lambda) := \sum_{j = 1}^{n} \gamma(\overline{\lambda})_{j} \otimes e_j.
\end{equation}
 From here it is not difficult to compute the matrix representation for $K_\lambda (z)$ in the standard $\{ e_i \}_{j = 1}^{n}$ basis for $\mathbb{C} ^n$ as
\begin{equation} K_\lambda (z) = \left[ \langle \gamma _i (\overline{\lambda}) , \gamma _j (\overline{z}) \rangle \right]_{1 \leqslant i, j \leqslant n}. \label{equation:kernel} \end{equation}
The alert reader might be worried about the verification of property \eqref{equation:condition}, the invertibility of $K_{\lambda}(z)$ for $\lambda, z \in \C_{+}$ or $\lambda, z \in \C_{-}$. Any model $\Gamma$, the Krein model in particular, will take the form in \eqref{model-anything}. Any other model $\widetilde{\Gamma}$ must then take the form
$$\widetilde{\Gamma}(\lambda) := \sum_{j = 1}^{n} \widetilde{\gamma}(\overline{\lambda})_{j} \otimes e_j,$$
where
$$\widetilde{\gamma}(\overline{\lambda})_{j} = \sum_{k = 1}^{n} c_{k, j}(\lambda) \gamma(\overline{\lambda})_{k}.$$
and the matrix $C_{\lambda} := (c_{k, j}(\lambda))_{k, j}$ is invertible. One can now check that
$$\widetilde{K}_{\lambda}(z) = C_z K_{\lambda}(z) C_{\lambda}^{*}.$$ So, if, with the Krein model we have $K_{\lambda}(z)$ is invertible for $\lambda, z \in \C_{+}$ or $\lambda, z \in \C_{-}$, then with any other model will also satisfy this property.

This Grauert's trick will help us avoid dealing with the self-adjoint extensions and the resolvents in Krein's formula \eqref{K-form}, which, as mentioned earlier, can be difficult to compute.
\end{Remark}

\section{Examples}

\subsection{Differentiation} \label{diff-ex}
Consider the simple differential operator $T f = i f'$ defined densely on $L^2[-\pi, \pi]$ with domain $\mathscr{D}(T)$ the Sobolev space of absolutely continuous functions $f$ on $[-\pi, \pi]$ with $f' \in L^2[-\pi, \pi]$ and $f(-\pi) = f(\pi) = 0$. Simple integration by parts will show that $T$ is symmetric and closed. Furthermore, $\mathscr{D}(T^{*})$ consists of the absolutely continuous functions $f$ on $[-\pi, \pi]$ with $f' \in L^2[-\pi, \pi]$. This is quite standard and can be found in many functional analysis books.

Observe, for any $\lambda \in \C \setminus \R$, that
$$\mbox{Ker} (T^{*} - \lambda I) = \{f\in \mathscr{D}(T^{*}): i f' =  \lambda f\} = \C e^{-i \lambda t}$$ and so the deficiency indices are both equal to one. Moreover, $T$ satisfies the simplicity condition \eqref{simple-defn} since
$$\bigvee_{\lambda \not \in \mathbb{R}} \mbox{Ker} (T^{*} - \lambda I) = \bigvee_{\lambda \not \in \mathbb{R}} e^{-i \lambda t} = L^2[-\pi, \pi]$$ via the Stone-Weierstrass theorem. Thus $T \in \mathcal{S}_{1}(L^2[-\pi, \pi])$. Define
$$\gamma(\lambda) := e^{-i \lambda t}, \quad \Gamma(\lambda) := \gamma(\overline{\lambda}) \otimes 1.$$
For $f \in L^2[-\pi, \pi]$ we have
$$\widehat{f}(\lambda) = \Gamma(\lambda)^{*} f = \int_{-\pi}^{\pi} f(t) e^{i \lambda t} dt.$$ Thus our Hilbert space of analytic functions $\mathcal{H}(\Gamma)$ is one of the classical Paley-Wiener spaces \cite{DB}. From here it follows that $T$ is unitarily equivalent to multiplication by the independent variable on the Paley-Wiener space. A computation will show that
$$K_{\lambda}(z) = \int_{-\pi}^{\pi} e^{-i z t} e^{i \overline{\lambda} t} d t = \frac{2 \sin(\pi (z - \overline{\lambda}))}{z - \overline{\lambda}}.$$

\subsection{Double differentiation} \label{double-diff}
Now consider the double differentiation operator $T f = -f''$ initially defined on the set $C_{0}^{\infty}(0, \infty)$ (smooth functions with compact support in $(0, \infty)$) and extend the domain of $T$ to the closure of $C_{0}^{\infty}(0, \infty)$ in the norm $\|f''\|_{L^2}$ -- which will be some Sobolev space. This domain $\mathscr{D}(T)$ is clearly dense in $L^2(0, \infty)$ and a simple computation with integration by parts will show that $T$ is symmetric and closed.

Note that $\mathscr{D}(T^{*})$ contains $C^{\infty}(0, \infty) \cap L^2(0, \infty)$ and moreover, $T$ has deficiency indices equal to one since
$$\mbox{Ker} (T^{*} - \lambda) = \{f \in \mathscr{D}(T^{*}): -f'' = \lambda f\} = \C e^{\pm i \sqrt{\lambda} t}.$$
Note that, depending on whether $\Im \lambda > 0$ or $\Im \lambda < 0$, only \textit{one} of the solutions $e^{\pm i \sqrt{\lambda} t}$ will belong to $L^2(0, \infty)$. Denote this solution by $$e^{i \epsilon_{\lambda} \sqrt{\lambda} t},$$
where $\epsilon_{\lambda} =  1$ if $\Im \lambda > 0$ and $\epsilon_{\lambda} = -1$ if $\Im \lambda < 0$.  Furthermore, one can check, using duality, that
$$\bigvee_{\lambda \not \in \mathbb{R}} \mbox{Ker} (T^{*} - \lambda) =  L^2(0, \infty)$$ and so $T \in \mathcal{S}_{1}(L^2(0, \infty))$. As in the previous example, we can define
$$\gamma(\lambda) := e^{i \epsilon_{\lambda} \sqrt{\lambda} t}, \quad \Gamma(\lambda) := \gamma(\overline{\lambda}) \otimes 1.$$
From here, for $f \in L^2(0, \infty)$,
$$\widehat{f}(\lambda) = \int_{0}^{\infty} f(t) e^{-i \epsilon_{\lambda} \sqrt{\lambda} t} dt.$$
The kernel function is
$$K_{\lambda}(z) = \int_{0}^{\infty} e^{- i \epsilon_{z} \sqrt{z} t} e^{i \epsilon_{\lambda} (\overline{\lambda} ) ^{\frac{1}{2}} t} dt = \frac{- i}{\epsilon_{z} \sqrt{z} - \epsilon_{\lambda} (\overline{\lambda} ) ^{\frac{1}{2}}}.$$

\subsection{Sturm-Liouville operators} \label{SL-ex}
In this example we will consider second-order Sturm-Liouville differential operators on intervals $I = [a,b] \subset \mathbb{R}$. A good reference for this is \cite{Naimark}. In particular, this will include Schr\"{o}dinger operators as a special case. Suppose that $p\geqslant 0$, and $q$ are real-valued functions on $I$ such that $ 1/p$ and $q$ are locally $L ^1$ functions on $(a,b)$, i.e.,  they belong to $L ^1$ of any compact subset of $(a,b)$. Define the dense domain \[ \mathscr{D} ( H (p, q, I) ^* ) := \left\{ f \in L^2 (I) \ | \ f , pf' \in L^1 _{loc} (I) \ \mbox{and} \ -(pf')' + qf \in L^2 [a,b] \right\},  \]
and then define $$H (p,q ,I ) ^* f = -(pf')' + qf, \quad f \in \mathcal{D} (H (p,q,I) ^* ).$$
 The theory of \cite[Section 17]{Naimark} shows that
 $$H (p,q, I) := (H (p,q,I) ^* ) ^*$$  belongs to $\mathcal{S}_{n}(L^2[a, b])$, where $n$ is either  $0, 1$, or $2$, depending on the properties of $p$ and $q$. Furthermore, $ H (p,q,I) ^* $ is its adjoint. Although it is non-trivial, it can be shown that $H(p,q,I)$ is simple whenever it has indices $(1,1)$ or $(2,2)$ \cite{Gil1, Gil2}, and hence $H(p,q,I)$ is either simple or self-adjoint. Recall here that any closed symmetric linear operator with indices $(0,0)$ is self-adjoint.

Fix an interior point $x_0 \in I$ and given any $z\in \mathbb{C}$ let $u_z$ and $v_z$ be solutions to the ordinary differential equation,
$$ -(pf')' + qf = z f, $$
which satisfy $$ \left( \begin{array}{cc} u_z (x_0 )  & p(x_0) u_z ' (x_0)  \\
v_z (x_0)  &  p(x_0) v_z ' (x_0 )  \end{array} \right) = \left( \begin{array}{cc} 1 & 0 \\ 0 & 1 \end{array} \right). $$
Using the method of Picard iterates, typically used to prove the existence-uniqueness theorem for ordinary differential equations, it is not difficult
to show that for any fixed $x \in (a,b)$, the solutions $u_z (x), v_z (x)$ are entire functions of $z$ (see for example \cite[Section 2.3]{Hille} or \cite[pgs. 51-56]{Naimark}). It follows that whenever $u_z, v_z$ belong
 to $L^2 (I)$, they are entire $L^2 (I)$-valued functions. Now suppose that $H(p,q,I)$ has deficiency indices $(2,2)$.
This happens, for example, if both $a,b$ are finite and $q, 1/p \in L^1 [a,b]$  \cite[Section 17]{Naimark}. In this case both $u_\lambda , v_\lambda$ belong to $L^2 (I)$ for all $\lambda \in \mathbb{C} $ (when $\lambda \in \mathbb{R}$ this is not obvious, but still true \cite[Theorem 4, Section 19.4]{Naimark}) and it follows that for any $z \in \mathbb{C} $,
$$\mbox{Ker} (H(p,q,I) ^* - z I ) = \bigvee \{ u_z, v_z \},$$  and so
we can define $\gamma _1 (z) = u_z$, $\gamma _2 (z) = v_z$, and then for each $z \in \mathbb{C} \setminus \mathbb{R}$, if $\{ e_i \} _{i=1} ^2$ is the standard orthonormal basis of $\mathbb{C} ^2$,
$\Gamma(z) : \mathbb{C} ^2 \rightarrow \mbox{Ker} (H (p,q , I) ^* - \overline{z} I )$ defined by
 \[ \Gamma (z) := \gamma_1 (\overline{z} ) \otimes e_1 + \gamma _2 (\overline{z} ) \otimes e_2, \] is a valid choice of model for $H(p,q,I)$.   Moreover, it follows from equation
(\ref{equation:kernel}) that $\mathcal{H} (\Gamma)$ has reproducing kernel: $$ K ^I _\lambda (z) = \left( \begin{array}{cc}  \int _I  u_{\overline{\lambda}} (x)  \overline{ u _{\overline{z}} (x) } dx    &
\int _I u _{\overline{\lambda}} (x)  \overline{ v_{\overline{z}} (x)} dx  \\  \int _I  v_{\overline{\lambda}} (x) \overline{ u _{\overline{z}} (x) } dx  &
\int _I  v _{\overline{\lambda}} (x)   \overline{ v_{\overline{z}}  (x) } dx   \end{array}  \right). $$

For a concrete example consider $p=1$ and $q(t) = V(t) = \frac{1}{2t^2}$. Then $$H _V  := H (1, V , [0,b] )$$ is a symmetric operator acting on its appropriate domain in $L^2 [0, b]$. We will consider the two cases where (i) $b < \infty$ and (ii) $b = \infty$. For the first case, as discussed above, the theory of \cite{Naimark} implies that $H_V $ has deficiency indices $(2,2)$ and the above calculations apply. We will, however, compute the
deficiency indices and subspaces directly for this example, and obtain a model for $H_V $ with explicit formulas in terms of Hankel and Bessel functions.
Choosing $\lambda \in \mathbb{C} \setminus \mathbb{R}$, two linearly independent solutions to the differential equation \[ -f'' (t)  +\frac{1}{2t^2} f(t) = \lambda f(t) \] are $$u _\lambda (t) = \sqrt{t} H _{\sqrt{3}/4} ^{(1)} (\sqrt{\lambda} t )$$ and
$$v_\lambda (t) =  \sqrt{t} H _{\sqrt{3}/4} ^{(2)} (\sqrt{\lambda} t ),$$ where the $H ^{(i)}, i = 1, 2,$ are Hankel functions of the first and second kind \cite{MR0167642}, and $\sqrt{\lambda}$ is chosen to be such that $ 0 < \mbox{arg} (\sqrt{\lambda}) < \frac{\pi}{2}$ when $\lambda \in \mathbb{C} _+$ and $ \pi < \mbox{arg} (\sqrt{\lambda}) < \frac{5\pi}{2}$ when $\lambda \in \mathbb{C} _-$.

One can verify \cite{MR1349110} that the Bessel-$J$ and Bessel-$Y$ functions behave asymptotically as
$$ J_\nu (t) \sim \frac{1}{\Gamma (\nu +1)} \left( \frac{t}{2} \right) ^\nu  \ \ \ \
Y_\nu (t) \sim \frac{-\Gamma (\nu )}{\pi } \left( \frac{2}{t} \right) ^\nu, $$ for small $t$. Here $\Gamma$ denotes the Euler gamma function. Since the Hankel functions satisfy $H_\nu ^{(j)} = J_\nu + (-1) ^{j +1} i Y _\nu$, it follows that in the case (i) where $b< \infty$, both solutions $v_\lambda , u_\lambda$ belong to $L^2 [0, b]$, and hence they both belong to $\mbox{Ker} (H_V ^* - \lambda I )$. We conclude that the deficiency indices of $H_V$ are $(2,2)$. Both solutions $u_\lambda$ and
$v_\lambda$ are analytic as functions of $\lambda$ in $\mathbb{C} \setminus \mathbb{R}$, and so we can choose $\gamma _1 (\lambda) = u_{\lambda}$ and $\gamma _2 (\lambda) = v_{\lambda}$. Then if $\{ e_k \}_{k = 1}^{2}$ denotes the standard
basis for $\mathbb{C} ^2$,
$$
 \Gamma (\lambda ) = \gamma _1 (\overline{\lambda}) \otimes e_1 + \gamma _2 (\overline{\lambda}) \otimes e_2, $$ is a valid choice of model for $H_V$.

By equation (\ref{equation:kernel}), the reproducing kernel for $\mathcal{H} (\Gamma )$ is
$$ K_\lambda (z) = \left( \begin{array}{cc} \int _0 ^b t H_{\frac{\sqrt{3}}{4}} ^{(1)} \left( (\overline{\lambda} ) ^{\frac{1}{2}} t \right) \overline{ H_{\frac{\sqrt{3}}{4}} ^{(1)} \left( (\overline{z} ) ^{\frac{1}{2}} t \right)} dt   & \int _0 ^b t
 H_{\frac{\sqrt{3}}{4}} ^{(1)} \left( (\overline{\lambda} ) ^{\frac{1}{2}} t \right) \overline{ H_{\frac{\sqrt{3}}{4}} ^{(2)} \left( (\overline{z} ) ^{\frac{1}{2}} t \right)} dt \\
 \int _0 ^b t
 H_{\frac{\sqrt{3}}{4}} ^{(2)} \left( (\overline{\lambda} ) ^{\frac{1}{2}} t \right) \overline{ H_{\frac{\sqrt{3}}{4}} ^{(1)} \left( ( \overline{z} ) ^{\frac{1}{2}} t \right)} dt & \int _0 ^b t
 H_{\frac{\sqrt{3}}{4}} ^{(2)} \left( (\overline{\lambda} ) ^{\frac{1}{2}} t \right) \overline{ H_{\frac{\sqrt{3}}{4}} ^{(2)} \left( (\overline{z} ) ^{\frac{1}{2}} t \right) } dt \end{array}\right).
$$

Now consider the second case where $b = \infty$. The Hankel functions have the asymptotic behavior
$$ H_\nu ^{(j)} (\sqrt{\lambda t} ) \sim \sqrt{\frac{2}{\pi \sqrt{\lambda} t }} e^{(-1)^{j+1} i \left( \sqrt{\lambda t} - \frac{\pi \nu}{2} - \frac{\pi}{4} \right)}$$
as $t \rightarrow \infty$ \cite{MR0167642}. It follows that for any $\lambda \in \mathbb{C} \setminus \mathbb{R}$, one of the solutions $u_\lambda , v_\lambda$ is square integrable on $[0 , \infty)$ and one is not. More precisely,
$u_\lambda $ is square integrable when $\lambda \in \mathbb{C} _+$ while $v_\lambda $ is square integrable if $\lambda \in \mathbb{C} _-$.  This shows
that in this case $H_V$ has deficiency indices $(1,1)$, and so if we define $$ \gamma (\lambda ) = \left\{ \begin{array}{cc} u_\lambda & \lambda \in \mathbb{C} _+ \\
v_\lambda & \lambda \in \mathbb{C} _-  \end{array} \right. , $$ and $\Gamma (\lambda ) := \gamma (\overline{\lambda} ) \otimes e_1$, then $\Gamma$ is a model for $H_V$, and the reproducing kernel for $\mathcal{H} (\Gamma)$ is
$$ K _\lambda (z) = \left\{ \begin{array}{cc} \int _0 ^\infty t H_{\frac{\sqrt{3}}{4}} ^{(1)} \left( (\overline{\lambda} ) ^{\frac{1}{2}} t \right) \overline{ H_{\frac{\sqrt{3}}{4}} ^{(1)} \left( (\overline{z} ) ^{\frac{1}{2}} t \right)} dt & \lambda, z  \in \mathbb{C} _+ \\
\int _0 ^\infty t H_{\frac{\sqrt{3}}{4}} ^{(1)} \left( (\overline{\lambda} ) ^{\frac{1}{2}} t \right) \overline{ H_{\frac{\sqrt{3}}{4}} ^{(2)} \left( (\overline{z} ) ^{\frac{1}{2}} t \right)} dt & \lambda \in \mathbb{C} _+ \ \  z \in \mathbb{C} _- \\
\int _0 ^\infty t H_{\frac{\sqrt{3}}{4}} ^{(2)} \left( (\overline{\lambda} ) ^{\frac{1}{2}} t \right) \overline{ H_{\frac{\sqrt{3}}{4}} ^{(1)} \left((\overline{z} ) ^{\frac{1}{2}} t \right)} dt & \lambda \in \mathbb{C} _- \ \   z \in \mathbb{C} _+ \\
\int _0 ^\infty t H_{\frac{\sqrt{3}}{4}} ^{(2)} \left( (\overline{\lambda} ) ^{\frac{1}{2}} t \right) \overline{ H_{\frac{\sqrt{3}}{4}} ^{(2)} \left((\overline{z} ) ^{\frac{1}{2}} t \right)} dt & \lambda, z \in \mathbb{C} _-  \\ \end{array} \right.$$

\subsection{Unbounded Toeplitz operators} \label{Toeplitz-ex}
Let $H^2$ denote the classical Hardy space of the open unit disk $\mathbb{D} := \{|z| < 1\}$ with inner product
\begin{equation} \label{H2-inner}
\langle f, g \rangle := \frac{1}{2 \pi} \int_{0}^{2 \pi} f(e^{i \theta}) \overline{g(e^{i \theta})} d \theta.
\end{equation}
Note how we equate, as is customary, a Hardy space function (analytic on $\mathbb{D}$) with its almost everywhere define $L^2(\partial \mathbb{D})$ radial boundary values on the unit circle $\partial \mathbb{D}$. Note that $H^2$ is a reproducing kernel Hilbert space with Cauchy kernel
 $$k_z(\zeta) = \frac{1}{1 - \overline{z} \zeta},$$
  i.e.,
$$f(z) = \langle f, k_{z} \rangle_{H^2}, \quad \forall f \in H^2, z \in \mathbb{D}.$$

Let $H^{\infty}$ denote the bounded analytic functions on $\mathbb{D}$ and
$N^{+}$ denote the \emph{Smirnov functions}, i.e., the algebra of analytic functions $f$ on $\mathbb{D}$ which can be written as $f = g/h$, where $g, h \in H^{\infty}$ and $h$ is an outer function. We refer the reader to the well-known texts \cite{Duren, Garnett} for a reference on all of this.

By a result of Sarason \cite{Sarason}, one can write each $g \in N^{+}$ as
\begin{equation} \label{gisba}
g = \frac{b}{a}; \quad a, b \in H^{\infty}, a(0) > 0, \ \mbox{$a$ outer}, \  |a(e^{i \theta})|^2 + |b(e^{i \theta})|^2 = 1 \; \; \mbox{a.e.-$\theta$}.
\end{equation}
  If $g$ is a rational function then so are $a$ and $b$. Since $a$ is an outer function, the set $a H^2$ is dense in $H^2$ and one can define a Toeplitz operator $T_g$ on $\mathscr{D}(T_g) = a H^2$ by
  $$T_{g} f = g f.$$
  In \cite{Sarason} it is shown that $T_g$ is a densely defined closed operator.
  If $g$ is also real-valued almost everywhere on $\partial \mathbb{D}$, then, by the definition of the inner product on $H^2$ from \eqref{H2-inner}, $T_{g}$ is a closed symmetric operator.

Let $N^{+}_{\R}$ denote the Smirnov functions which are real valued on the unit circle. Helson \cite{Helson, Helson2} shows that $g \in N^{+}_{\R}$ if and only if there are inner functions $p, q$ with $p - q$ outer such that
\begin{equation} \label{g-Helson}
g = i \frac{q + p}{q - p}.
\end{equation}
The inner functions $p$ and $q$ in the above representation are unique up to constant factors.
Furthermore, the deficiency  indices for $T_{g}$, $g \in N^{+}_{\R}$, are finite if and only if $g$ is a rational function and Helson was able to construct examples of $g$ for which the defect indices are any pair $(m, n)$, $m, n \in \mathbb{N} \cup\{\infty\}$.  Indeed, let $\nu$ be an real signed atomic measure on $\partial \mathbb{D}$ with $m$ point masses which are positive and $n$ which are negative. The function
$$g(z) := i \int \frac{\zeta + z}{\zeta - z} d \nu(\zeta)$$
belongs to $N^{+}_{\R}$ and the deficiency indices of $T_{g}$ are $(m, n)$.

If $g$ takes the form \eqref{g-Helson} then
$$g - i = p \frac{2 i}{q - p}, \quad g + i = q \frac{2 i}{q - p}$$
and so $p$ is the inner factor of $g - i$ while $q$ is the inner factor of $g + i$. Thus
$$\mbox{Rng}(T_{g} - i I)^{\perp} = (p H^2)^{\perp}, \quad \mbox{Rng}(T_{g} + i I)^{\perp} = (q H^2)^{\perp}.$$
It is well-known that these spaces have dimension equal to the order of the inner functions $p$ and $q$. Hence,
$T_{g} \in \mathcal{S}_{n}(H^2)$ when $g \in N^{+}_{\R}$ and $p$ and $q$ are inner functions of equal order.

\begin{Remark}
This brings up the interesting question as to when, for two inner functions $p$ and $q$ of equal order, the difference $p - q$ is outer. Certainly when $p$ and $q$ are Blaschke products of order $n$ the condition $p - q$ is outer implies that $p$ and $q$ do not share any zeros. One might be tempted to believe the converse. Unfortunately this is not true. Take two different $n$-th roots on unity $\zeta, \xi$ and $0 < a < 1$. The inner functions
$$p(z) = \left(\frac{z  - a \zeta}{1 - \overline{a \zeta} z}\right)^n, \quad q(z) = \left(\frac{z  - a \xi}{1 - \overline{a \xi} z}\right)^n$$ have distinct zeros $a \zeta \not = a \xi$. But $p(0) - q(0) = 0$ and so $p - q$ is not outer.
\end{Remark}

From the well-known identity $$T_{g}^{*} k_{z} = \overline{g(z)} k_{z},$$  we see that if  $T_{g} \in \mathcal{S}_{1}(H^2)$ then $g$ must be univalent. Moreover, we have
$$\mbox{Ker} (T_{g}^{*} - \lambda I) = \mathbb{C} k_{g^{-1}(\overline{\lambda})}.$$
Thus, as in our previous examples,
$$\gamma(\lambda) = \frac{1}{1 - g^{-1}(\overline{\lambda}) z}, \quad \Gamma(\lambda) = \gamma(\overline{\lambda}) \otimes 1, \quad K_{\lambda}(z) = \frac{1}{1 - \overline{g^{-1}(\lambda)} g^{-1}(z)}.$$
In this case the corresponding $H^{2}(\Gamma)$ space is the set of functions of the form
$$\widehat{f}(\lambda) = \langle f, k_{g^{-1}(\lambda)} \rangle = f(g^{-1}(\lambda)), \quad f \in H^2(\mathbb{D}),$$ and so $H^2(\Gamma)$ is the Hardy space $H^{2}(g(\mathbb{D}))$, where, since $g$ is real on the circle, will be the Hardy space of a certain slit domain. The norming point of $H^2(g(\mathbb{D}))$ will be $g(0)$. See \cite{A-F-R} for more on Hardy spaces of a slit domains.

For a vector-valued example, consider the Toeplitz operator $T_{g^2}$, where $g \in N^{+}_{\R}$ and $g$ is univalent. Then
$$\mbox{Ker} (T^{*}_{g^2} - \lambda I) = \bigvee \left\{k_{g^{-1}((\overline{\lambda} ) ^{\frac{1}{2}})}, k_{g^{-1}(-(\overline{\lambda} ) ^{\frac{1}{2}})}\right\}.$$
Here
$$\gamma_1(\lambda) :=  \frac{1}{1 -  g^{-1}((\overline{\lambda} ) ^{\frac{1}{2}}) z}, \quad \gamma_2(\lambda) := \frac{1}{1 -  g^{-1}(-(\overline{\lambda} ) ^{\frac{1}{2}}) z}$$
and if $e_1 = (1, 0), e_2 = (0, 1)$ are the standard basis vectors for $\C^2$, then
$$\Gamma(\lambda) = \gamma_{1}(\overline{\lambda}) \otimes e_1 + \gamma_{2}(\overline{\lambda}) \otimes e_2.$$
In this case the kernel function turns out to be
$$K_{\lambda}(z) =
\left(
\begin{array}{cc}
  \frac{1}{1 -  g^{-1}((\overline{\lambda} ) ^{\frac{1}{2}}) g^{-1}(\sqrt{z})}   &  \frac{1}{1 -  g^{-1}((\overline{\lambda} ) ^{\frac{1}{2}}) g^{-1}(-\sqrt{z})}  \\
   \frac{1}{1 -  g^{-1}(-(\overline{\lambda} ) ^{\frac{1}{2}}) g^{-1}(\sqrt{z})} &    \frac{1}{1 -  g^{-1}(\sqrt{-\overline{\lambda}}) g^{-1}(-\sqrt{z})}\\
\end{array}
\right),$$
and the corresponding $H^2(\Gamma)$ space is the set of functions of the form
$$\widehat{f}(\lambda) = (f(g^{-1}(\sqrt{\lambda})), f(g^{-1}(-\sqrt{\lambda}))), \quad f \in H^2.$$

\subsection{Multiplication by the independent variable on a Lebesgue space}
Suppose $\mu$ is a positive Borel measure on $\R$ such that
$$\mu(\R) = \infty \quad \mbox{and} \quad \int \frac{1}{1 + x^2} d \mu(x) < \infty.$$
It is well-known that the operator $(M^{\mu}f)(x) = x f(x)$, defined densely on
$$\mathscr{D}(M^{\mu}) := \{f \in L^2(\mu) : x f \in L^2(\mu)\},$$ is self adjoint.  Consider the operator $M_{\mu}$ defined as $M^{\mu}$ restricted to
$$\mathscr{D}(M_{\mu}) := \left\{f \in L^2(\mu): x f \in L^2(\mu), \int f d\mu = 0\right\}.$$
Notice the difference between $M^{\mu}$, which is self-adjoint, and $M_{\mu}$, which is symmetric but not self-adjoint.

\begin{Proposition} \label{Mmu}
For a positive Borel measure $\mu$ on $\R$ with $\int \frac{1}{1 + x^2} d \mu(x) < \infty$, the operator $M_{\mu}$ is densely defined if and only if $\mu(\R) = \infty$.
\end{Proposition}

We start with the following technical lemma from functional analysis.

 \begin{Lemma} \label{TLmu}
Suppose $B$ is a dense linear manifold in $L^2(\mu)$ and $\ell$ is a linear functional defined on $B$ such that there exists a sequence of unit vectors $\{f_n\}_{n \geqslant 1}$ in $B$ such that $\ell(f_n) \not = 0$ and $\ell(f_n) \to \infty$ as $n \to \infty$. Then the linear manifold
$$\{f \in B: \ell(f) = 0\}$$
is dense in $L^2(\mu)$.
\end{Lemma}

\begin{proof}
Suppose $\lambda$ is a \emph{bounded} linear functional on $L^2(\mu)$ with $$\lambda(\{f: \ell(f) = 0\}) = 0.$$ Then
$$\lambda\left(f  - \frac{\ell(f)}{\ell(f_n)} f_n \right) = 0 \quad \forall f \in B.$$ As $n \to \infty$ we use the fact that $\|f_n\| = 1$ and $\ell(f_n) \to \infty$ to see that $\lambda(f) = 0$ for all $B$ and so, by and Hahn-Banach theorem, $\lambda \equiv 0$.
\end{proof}

\begin{proof}[Proof of Proposition \ref{Mmu}]Let us show that
$$\mathscr{D}(M_{\mu}) = \left \{f \in L^2(\mu): t f \in L^2(\mu), \int f(t) d \mu(t) = 0\right\}
$$
 is dense in $L^2(\mu)$ if and only if $\mu(\R) = \infty$. Suppose that $\mu(\R) < \infty$. Then $\C \subset L^2(\mu)$ and $\mathscr{D}(M_{\mu}) \subset \C^{\perp}$ and so $\mathscr{D}(M_{\mu})$ is not dense in $L^2(\mu)$.
 Now suppose that $\mu(\R) = \infty$. If $f \in L^2(\mu)$ and $t f \in L^2(\mu)$ then
 \begin{align*}
 \int |f(t)|  d\mu(t) & = \int |f(t)| (1 + |t|) \frac{1}{1 + |t|} d \mu(t)\\
 &  \leqslant \left( \int |f(t)|^2 (1 + t^2) d \mu(t)\right)^{1/2} \left(\int \frac{1}{1 + t^2} d \mu(t)\right)^{1/2}\\
 & < \infty.
 \end{align*}
 Thus the linear functional
 $$f \mapsto \int f(t) d \mu(t)$$ is defined on $\{f \in L^2(\mu): t f \in L^2(\mu)\}$. This last set is dense in $L^2(\mu)$ since it contains the smooth functions with compact support. But since $\mu(\R) = \infty$, this linear functional satisfies the hypothesis of Lemma \ref{TLmu} (Indeed just take $f_{n} \in C_{0}^{\infty}(\R)$, $f_n \geq 0$ with $f_{n} = 1$ on $[-N, N]$ in Lemma \ref{TLmu}) and so $\mathscr{D}(M_{\mu})$ is dense in $L^2(\mu)$.
\end{proof}

For $g \in \mathscr{D}(M_{\mu})$ observe that
$$0 = \int g d\mu = \int g(x) (x - \lambda) \overline{\frac{1}{x - \overline{\lambda}}} d \mu(x)$$
and so $\frac{1}{x - \overline{\lambda}} \in \mbox{Rng}(M_{\mu} - \lambda I)^{\perp}$. A little exercise will show that indeed
$$\mbox{Rng}(M_{\mu} - \lambda I)^{\perp} = \C \frac{1}{x - \overline{\lambda}}$$
and thus $M_{\mu} \in \mathcal{S}_{1}(L^2(\mu))$.
Notice that
$$\gamma(\lambda) = \frac{1}{x - \overline{\lambda}}, \quad \Gamma(\lambda) = \gamma(\overline{\lambda})  \otimes 1$$ and thus
$$K_{\lambda}(z) = \int \frac{1}{(x - z)(x - \overline{\lambda})} d \mu(x)$$ and
$$\widehat{f}(\lambda) = \int \frac{f(x)}{x - \lambda} d \mu(x), \quad f \in L^2(\mu).$$
Thus $M_{\mu}$ is unitarily equivalent to $M_{\Gamma}$ (multiplication by the independent variable) on the space of Cauchy transforms of $L^2(\mu)$ functions.

We point out that there is a version of all this when $\mu$ is a positive matrix-valued measure on $\R$ which will be explored later on.

\subsection{Multiplication by the independent variable on a deBranges-Rovnyak space}  Let $H^2 _\mathcal{K}$ denote the Hardy space of analytic $\mathcal{K}-$valued functions on the upper half-plane $\mathbb{C} _+$. These are the analytic functions $f: \C_{+} \to \mathcal{K}$ such that
$$\sup_{y  > 0} \int_{-\infty}^{\infty} \|f(x + i y)\|_{\mathcal{K}}^{2} dx < \infty.$$
It is well known that these functions have non-tangential boundary values almost everywhere on $\R$ such that
$$\int_{-\infty}^{\infty} \|f(x)\|_{\mathcal{K}}^2 dx < \infty.$$
The above quantity determines an inner product $\langle \cdot , \cdot \rangle$ on $H^2_{\mathcal{K}}$. Given any contractive analytic $B (\mathcal{K} )-$valued function $\Theta$, one can define the de Branges-Rovnyak space $\mathscr{K}(\Theta)$ as the range of the operator
\begin{equation}\label{Rtheta}
R_\Theta := (I - T _\Theta T_\Theta^{*})^{1/2},
\end{equation}
where the inner product $ \langle \cdot , \cdot \rangle _\Theta$ on $\mathscr{K}(\Theta)$ is defined so that $R_\Theta$ acts as a co-isometry of $H^2 _\mathcal{K} $ onto $\mathscr{K}(\Theta)$.  In other words, if at least one of $f , g \in H^2 _\mathcal{K}$ are orthogonal to $\mbox{Ker} (R _\Theta )$ then
$$\langle R_\Theta f , R_\Theta g \rangle _\Theta = \langle f , g \rangle.$$
 In \eqref{Rtheta}, $$T_\Theta := P _{H^2 _\mathcal{K}} M _\Theta |H^2 _\mathcal{K}$$
is a Toeplitz operator, $M_\Theta$ is multiplication by $\Theta$ acting on $L^2 _\mathcal{K}$-the Hilbert space of $\mathcal{K}$-valued functions which are square integrable with respect to Lebesgue measure on $\mathbb{R}$, and $P_{H^2 _\mathcal{K}}$ is the orthogonal projection from $L^2 _\mathcal{K}$ onto $H^2 _\mathcal{K}$.
Note that $T_{\Theta}$ is a contraction and so $R_{\Theta}$ makes sense (and is a contraction). Note also that when
$\Theta$ is inner, i.e., $\Theta(x)$ is unitary for almost every $x \in \R$, then the deBranges-Rovnyak space
$\mathscr{K}(\Theta)$ becomes the classical model space $H^{2}_{\mathcal{K}} \ominus \Theta H^2_{\mathcal{K}}$ in the upper half plane \cite{Nik, Nik2, Niktr}.

If $\{ e_k \}_{k = 1}^{n}$ is an orthonormal basis for $\mathcal{K}$, then it follows that finite linear combinations of the  reproducing kernel vectors $$\delta _w ^{(j)} (z) := \frac{i}{2\pi}  \frac{1}{z-\overline{w}} e_j, \quad 1 \leqslant j \leqslant n, $$
 form a dense set in $H^2 _\mathcal{K}$. Since
$\mathscr{K}(\Theta)$ is contractively contained in $H^2 _\mathcal{K}$ (this follows because the operator $R_\Theta$ is a contraction), it follows that given any $1 \leqslant j \leqslant n$ and $z \in \mathbb{C} _+$, the point evaluation linear
functional $l ^{(j)} _z $ defined by \[ l ^{(j)} _z (f) = \langle f(z) , e_j \rangle _{\mathcal{K}} ,\] is well-defined and bounded on $\mathscr{K}(\Theta)$. From the Riesz representation theorem, there is a
point evaluation or reproducing kernel vector $\sigma ^{(j)} _z \in \mathscr{K}(\Theta)$ such that for any $h \in \mathscr{K}(\Theta)$, \[ \langle h(z) , e_j \rangle _{\mathcal{K}} =  l _z ^{(j)} (h) = \langle h , \sigma ^{(j)} _z \rangle _\Theta . \]
To compute $\sigma ^{(j)} _z$, consider the fact that if  $h = R_\Theta f \in \mathscr{K}(\Theta)$ for some $f \in H^2 _\mathcal{K}$, then
\begin{eqnarray}   \langle h(z) , e_j \rangle _{\mathcal{K}} & =  & \langle h , \delta ^{(j)} _z \rangle  = \langle R_\Theta f , \delta ^{(k)} _z \rangle \nonumber \\
& = & \langle f , R_\Theta ^* \delta _z ^{(j)} \rangle = \langle h , R_\Theta R_\Theta ^* \delta _z ^{(j)} \rangle _\Theta.\end{eqnarray}
This computation shows that \[ \sigma _z ^{(j)} = R_\Theta R_\Theta ^* \delta _z ^{(j)} = (I - T_\Theta T_{\Theta } ^* ) \delta _z ^{(j)}.\]
An easy calculation shows $T_\Theta ^* \delta _z ^{(j)} = \Theta (z) ^* \delta _z ^{(j)} $, and so it follows that
\[ \sigma _\lambda ^{(j)} (z) = \frac{i}{2\pi} \frac{ I   - \Theta (z) \Theta ^* (\lambda)}{z-\overline{\lambda}} e_j, \] and hence the reproducing kernel operator on $\mathscr{K}(\Theta)$ is
\begin{equation}  \label{dbrkernel}
\Delta  _w ^\Theta (z) := \frac{i}{2\pi} \frac{I - \Theta(z) \Theta (w) ^*}{z -\overline{w} }; \quad w,z \in \mathbb{C} _+ .
\end{equation}
Moreover, finite liner combinations of
$$\sigma _w ^{(j)} = \Delta _w ^\Theta e_j, \quad w \in \mathbb{C} _+, 1 \leqslant j \leqslant n,$$ are dense in $\mathscr{K}(\Theta)$.

Let $H ^\infty _\mathcal{K}$ be the Banach space of all bounded analytic $B (\mathcal{K})$-valued functions on $\mathbb{C} _+$ (with the supremum norm), and
$$\mathscr{B}_{\mathcal{K}} := \{f \in H^{\infty}_{\mathcal{K}}: \|f\|_{\infty} \leqslant 1\}$$
be the closed unit ball in $H^{\infty}_{\mathcal{K}}$.
A function $f \in \mathscr{B}_{\mathcal{K}}$ is an \emph{extreme point} of $\mathscr{B}_{\mathcal{K}}$ if $f$ does not belong to the interior of a line segment lying in $\mathscr{B}_{\mathcal{K}}$. Equivalently, $f \in \mathscr{B}_{\mathcal{K}}$ is extreme if $f \pm g \in \mathscr{B}_{\mathcal{K}}$, where $g \in H^{\infty}_{\mathcal{K}}$, implies $g \equiv 0$.

If $\theta$ is a contractive matrix-valued analytic function on the open unit disk $\mathbb{D}$ and $\vec{x} \in \C^{n}$, we say that $\theta \vec{x}$ has an \emph{angular derivative in the sense of Carath\'{e}odory} at $\zeta \in \partial \mathbb{D}$ if $\theta \vec{x}$ has a non-tangential limit $\theta(\zeta) \vec{x}$ at $\zeta$, $\|\theta(\zeta) \vec{x}\| = \|\vec{x}\|$, and the non-tangential limit of $\theta' \vec{x}$ exists at $\zeta$. Existence of angular derivatives relates to non-tangential limits of functions in model and deBrances-Rovnyak spaces \cite{AC, Martin-uni, Poltoratskii, Sarason-dB}.

From  \cite{Martin-uni} we know the following.

\begin{Theorem} \label{Martin-Z}
If $\Theta \in \mathscr{B}_{\C^n}$ is an extreme point then $Z_{\Theta} f := z f$, defined on
$$\mathscr{D}(Z_{\Theta}) := \{f \in \mathscr{K}(\Theta): z f \in \mathscr{K}(\Theta)\},$$
is a closed symmetric operator with deficiency indices $(n, n)$. Moreover, $Z_{\Theta}$ is densely defined if and only if $(\Theta \circ b^{-1})\vec{k}$ does not have a finite angular derivative at $z = 1$ for any $\vec{k} \in \C^n$. In this case $Z_{\Theta} \in \mathcal{S}_{n}(\mathscr{K}(\Theta))$.
\end{Theorem}

For the sake of simplicity, let us discuss, as in our previous examples, the model space for $Z_{\Theta}$ when $Z_{\Theta} \in \mathcal{S}_{1}(\mathscr{K}(\Theta))$, i.e., $\Theta$ is scalar valued. In this case the kernel functions for $\mathscr{K}(\Theta)$ are
$$\Delta^{\Theta}_{\lambda}(z) = \frac{i}{2 \pi} \frac{1 - \Theta(z) \overline{\Theta(\lambda)}}{z - \overline{\lambda}}, \quad z, \lambda \in \C_{+}.$$ Notice that
\begin{equation} \label{e-values-Z}
Z_{\Theta}^{*} \Delta^{\Theta}_{\lambda} = \overline{\lambda} \Delta^{\Theta}_{\lambda}, \quad \lambda \in \C_{+}.
\end{equation}
Thus
$$\mbox{Rng}(Z_{\Theta} - \lambda I)^{\perp} = \C \Delta^{\Theta}_{\lambda}, \quad \lambda \in \C_{+}.$$ To complete the picture we need to compute $\mbox{Rng}(Z_{\Theta} - \lambda I)^{\perp}$ when $\lambda \in \C_{-}$. Notice the complication here in using the identity in \eqref{e-values-Z} since $\Delta^{\Theta}_{\lambda}$ is not defined when $\lambda \in \C_{-}$.

By results from \cite{Martin1} there exists a conjugation $C_{\Theta}: \mathscr{K}(\Theta) \to \mathscr{K}(\Theta)$, i.e.,  $C _\Theta$ is an involutive, isometric and conjugate linear operator defined by $C_\Theta f = \Theta \circ *$. Here $* : H^2 (\C _+ ) \rightarrow H^2 (\C _- )$ is the involutive, onto, isometric and anti-linear map defined by $* f(z) =: f^* (z)$, where $f^* (z) =: \overline{f (\overline{z} )}$ is defined to be the function in $H^2 (\C _-)$ whose non-tangential boundary values are given by $\overline{f (x) }$. Hence $$ (C_\Theta f) (z) = \Theta (z) \overline{f(\overline{z})} ,$$ and again this has to be interpreted as the function in $H^2 (\C _+ )$
whose non-tangential boundary values are equal to $\Theta (x) \ov{f(x)}$ almost everywhere.  Moreover, $C_{\Theta}$ maps $\mathscr{D}(Z_{\Theta})$ to itself, commutes with $Z_{\Theta}$, and satisfies
$$C_{\Theta} \mbox{Rng}(Z_{\Theta} - \lambda I)^{\perp} = \mbox{Rng}(Z_{\Theta} - \overline{\lambda} I)^{\perp}, \quad \lambda \in \C \setminus \R,$$ as well as
$$C_{\Theta} \Delta^{\Theta}_{\lambda} = \frac{1}{2 \pi i} \frac{\Theta(z) - \Theta(\lambda)}{z - \lambda}.$$ Therefore,
$$\mbox{Rng}(Z_{\Theta} - \lambda I)^{\perp}  = \C \{ C_{\Theta} \Delta^{\Theta}_{\overline{\lambda}} \}, \quad \lambda \in \C_{-}.$$ As in our previous examples, the model space for $Z_{\Theta}$ will be the space of functions of the form
$$\widehat{f}(\lambda) = \begin{cases} \langle f, \Delta^{\Theta}_{\lambda}\rangle_{\Theta} &\mbox{if } \lambda \in \C_{+} \\
\langle f, C_{\Theta} \Delta^{\Theta}_{\overline{\lambda}} \rangle_{\Theta} & \mbox{if } \lambda \in \C_{-}  \end{cases} =
\begin{cases} f(\lambda) &\mbox{if } \lambda \in \C_{+} \\
\overline{(C_{\Theta} f)(\overline{\lambda})} & \mbox{if } \lambda \in \C_{-}. \end{cases}$$

When $\Theta$ is an inner function, we can unpack this a bit further. In this case, as mentioned earlier, the deBranges-Rovnyak space $\mathscr{K}(\Theta)$ is the classical model space $(\Theta H^2(\C_{+}))^{\perp}$. Moreover, the inner product is the usual $L^2(\mathbb{R})$ inner product. For $f \in (\Theta H^2(\C_{+}))^{\perp}$ we have
$\widehat{f}(\lambda) = f(\lambda)$, $\lambda \in \C_{+}$. For $\lambda \in \C_{-}$ we have
\begin{align*}
\widehat{f}(\lambda) & = \langle f, C_{\Theta} \Delta^{\Theta}_{\overline{\lambda}} \rangle\\
& = \frac{1}{2 \pi i} \int_{-\infty}^{\infty} f(x) \overline{\left(\frac{\Theta(x)}{x - \overline{\lambda}}\right)}dx - \frac{1}{2 \pi i}\overline{\Theta(\overline{\lambda})} \int_{-\infty}^{\infty} f(x) \overline{\left(\frac{1}{x - \overline{\lambda}}\right)} dx\\
& = \frac{1}{2 \pi i} \int_{-\infty}^{\infty} \frac{f(x)\overline{\Theta(x)}}{x - \lambda} dx.
\end{align*}
Using Fatou's jump theorem and a similar computation as used in \cite[p.~85]{CR}, one can show that the non-tangential limits of $f/\Theta$ (from $\C_{+}$) are equal to the non-tangential limits of $\widehat{f}$ (from $\C_{-}$) almost everywhere on $\mathbb{R}$. Thus $\widehat{f}$ is a pseudo-continuation of $f/\Theta$ to $\C_{-}$. See also \cite{RS}.

We include $Z_{\Theta}$ in our list of examples since this operator will be closely related to the model operator on the Herglotz space we discuss later on -- and will also help us gain some additional information about the Livsic function.

\section{The main results}

For a fixed $T \in \mathcal{S}_{n}(H)$ modeled by $\Gamma$,  we have an associated $\mathcal{K}$-valued reproducing kernel Hilbert space $\mathcal{H}(\Gamma)$ of analytic functions on $\C \setminus \R$ such that $T$ is unitarily equivalent to $M_{\Gamma}$ (multiplication by the independent variable) on $\mathcal{H}(\Gamma)$. We also have a formula for the reproducing kernel $K_{\lambda}(z) = \Gamma(z)^{*} \Gamma(\lambda)$ for $\mathcal{H}(\Gamma)$. Our first theorem says that the reproducing kernel can be factored in a particular way, which, as we will see momentarily, involves the Livsic characteristic function.

\begin{Theorem} \label{T-main-vector-kernel}
Using the above notation we have
$$K_{\lambda}(z) = \Phi(z) \left(\frac{I - V(z) V(\lambda)^{*}}{1 - \overline{b(\lambda)} b(z)}\right) \Phi(\lambda)^{*},$$
where
$$  V(z) := b(z) \Phi (z) ^{-1} \Psi (z), $$
$$b(z) = \frac{z-i}{z+i}, \quad \Phi (z) := K_i (z) K_i (i) ^{-1/2},   \quad\Psi (z) := K_{-i} (z) K _{-i} (-i) ^{-1/2},$$  and
$\Phi$ and $V$ satisfy the following:
\begin{enumerate}
\item  The function $z \mapsto \Phi(z)$ is a $\mathcal{B}(\mathcal{K})$-valued meromorphic function on $\C \setminus \R$ which is analytic on $\C_{+}$.
\item  The function $z \mapsto \Psi(z)$ is a $\mathcal{B}(\mathcal{K})$-valued meromorphic function on $\C \setminus \R$ which is analytic on $\C_{-}$.
\item The function $z \mapsto V(z)$ is a $\mathcal{B}(\mathcal{K})$-valued meromorphic function on $\C \setminus \R$ such that $\|V(z)\| < 1$ for all $z \in \C_{+}$.
\end{enumerate}
\end{Theorem}

The skeptical reader might be wondering why factoring the reproducing kernel in this particular way is important.  This is answered by the following corollary.

\begin{Corollary}\label{Cor-main-Liv-ue}
For $T_1 \in \mathcal{S}_{n}(\mathcal{H}_1)$ and $T_2 \in \mathcal{S}_{n}(\mathcal{H}_2)$ let $V_1$ and $V_2$ be the corresponding operators from Theorem \ref{T-main-vector-kernel}. Then $T_1$ is unitarily equivalent to $T_2$ if and only if there are constant unitary operators $R, Q$ on $\mathcal{K}$ so that
$$V_1(z) = R V_2(z) Q, \quad z \in \C_{+}.$$
\end{Corollary}

What does all this have to do with Livsic's theorem?

\begin{Corollary} \label{VisW}
If $n < \infty$ and $T \in \mathcal{S}_{n}(\mathcal{H})$ then there are constant unitary matrices $R, Q$ so that
$$w_T(z) = R V(z) Q, \quad z \in \C_{+},$$ i.e.,  $w_T (z)$ and $V(z)$ are equivalent in terms of \eqref{RQ-equiv}.
\end{Corollary}

As to be expected, the proof will require some preliminary technical results.

\begin{Remark}Let us assume that we are working at the $\mathcal{H}(\Gamma)$ level and thus equate $f \in \mathcal{H}$ with $\widehat{f} \in \mathcal{H}(\Gamma)$. This will avoid the cumbersome $\widehat{f}$ notation in all of our calculations. With this understanding, we also note that we are now determining when $M_{\Gamma_1} \cong M_{\Gamma_2}$, where $\Gamma_j$ is a model for $T_j \in \mathcal{S}_{n}(\mathcal{H}_j)$  and $M_{\Gamma_j}$ is multiplication by the independent variable from Proposition \ref{H-Gamma2}.\end{Remark}

As we head towards our formula for $K_{\lambda}(z)$, let $P_{\pm i}$ be the orthogonal projections of $\mathcal{H}(\Gamma)$ onto
$$\{f \in \mathcal{H}(\Gamma): f(\pm i) = 0\}^{\perp}.$$ By the definition of reproducing kernel notice that
\begin{equation} \label{PH}
\{f \in \mathcal{H}(\Gamma): f(\pm i) = 0\}^{\perp} = \bigvee\{K_{\pm i}(z) a: a \in \mathcal{K}\}.
\end{equation}
Define
$$L, R: \mathcal{H}(\Gamma) \to \mathcal{H}(\Gamma)$$ by the formulas
$$L = \frac{1}{b} (I - P_i), \quad R = b (I - P_{-i}).$$
The alert reader might wonder why these operators $L$ and $R$ are actually defined on $\mathcal{H}(\Gamma)$.
Since $I - P_{\pm}$ is the orthogonal projection of $\mathcal{H}(\Gamma)$ onto $\{f \in \mathcal{H}(\Gamma): f(\pm i) = 0\}$ and
$$b(z) = 1 - \frac{2 i}{z - i}, \quad \frac{1}{b(z)} = 1 + \frac{2 i}{z + i},$$
the operators $R$ and $L$ become well-defined, and, by the closed graph theorem, bounded on $\mathcal{H}(\Gamma)$ once we verify the following lemma.

\begin{Lemma}
    If $f \in \mathcal{H} (\Gamma)$ and $f (w) =0$ for some $w\in \mathbb{C} \setminus \mathbb{R}$, then $\frac{f }{z-w} \in \mathcal{H} (\Gamma)$.
\end{Lemma}

\begin{proof}
    If $f (w) =0$ then $\Gamma (w) ^* f =0$ which implies $f \in \mbox{Rng} (T-w I)$. Thus $f = (T-w)g$ for some $g \in \mathscr{D} (T)$. Now,
    \begin{eqnarray} f (z) & = & \Gamma (z) ^* f \nonumber\\
    &= &\Gamma (z) ^* (T-w) g \nonumber \\
    & = & \Gamma (z) ^* \left( (T-z) + (z-w) \right) g \nonumber \\
    & = & (z-w) g (z), \end{eqnarray} which demonstrates that $\frac{f }{z-w} = g \in \mathcal{H} (\Gamma)$.
\end{proof}

Since $M_{\Gamma} f = z f$ on $\mathcal{H}(\Gamma)$ is symmetric with equal deficiency indices, then its \emph{Cayley transform} \cite{A-G}
$$C_{\Gamma} := (M_{\Gamma} - i I)(M_{\Gamma} + i I)^{-1},$$
is a partial isometry with initial space $(P_{-i} \mathcal{H}(\Gamma))^{\perp}$ and final space $(P_{i} \mathcal{H}(\Gamma))^{\perp}$. Moreover
$$C_{\Gamma} f = b f, \quad f \in (P_{-i} \mathcal{H}(\Gamma))^{\perp}.$$

\begin{Lemma}\label{RLnn}
$L^{*} = R$
\end{Lemma}

\begin{proof}
Note that $(L f)(-i) = 0 = (R f)(i)$ and so
$$\langle L f, P_{-i} g \rangle_{\mathcal{H}(\Gamma)} = 0 = \langle P_{i} f, R g \rangle_{\mathcal{H}(\Gamma)}.$$
 Also note that, via the above Cayley transform discussion, the map $f \mapsto b f$ is a partial isometry with initial space $(P_{-i} \mathcal{H}(\Gamma))^{\perp}$ and final space $(P_{i} \mathcal{H}(\Gamma))^{\perp}$.
Putting this all together we get
\begin{align*}
\langle Lf, g \rangle_{\mathcal{H}(\Gamma)} & = \langle L f, g - P_{-i} g \rangle_{\mathcal{H}(\Gamma)}\\
& = \langle b L f, b(g - P_{-i} g)\rangle_{\mathcal{H}(\Gamma)}\\
& = \langle f - P_{i} f, R g \rangle_{\mathcal{H}(\Gamma)}\\
& = \langle f, R g\rangle_{\mathcal{H}(\Gamma)},
\end{align*}
which proves $L^{*} = R$.
\end{proof}

\begin{Lemma}
For each $a \in \mathcal{K}$ and $\lambda, z \in \C \setminus \R$ we have
$$K_{\lambda}(z) a = \frac{(P_{i} K_{\lambda}(\cdot) a)(z) - \overline{b(\lambda)} b(z) (P_{-i} K_{\lambda}(\cdot) a)(z)}{1 - \overline{b(\lambda)} b(z)}.$$
\end{Lemma}

\begin{proof}
First let us compute $(L^{*} K_{\lambda}(\cdot) a)(z)$ and $(R K_{\lambda}(\cdot) a)(z)$. Indeed, for any $f$,
\begin{align*}
\langle f, L^{*} K_{\lambda}(\cdot) a\rangle_{\mathcal{H}(\Gamma)} & = \langle L f, K_{\lambda}(\cdot) a\rangle_{\mathcal{H}(\Gamma)}\\
& = \left\langle \frac{1}{b} (f - P_{i} f), K_{\lambda}(\cdot) a\right\rangle_{\mathcal{H}(\Gamma)}\\
& = \left\langle \frac{f(\lambda) - (P_{i} f)(\lambda)}{b(\lambda)}, a \right\rangle_{\mathcal{K}}\\
& = \frac{1}{b(\lambda)} \left\langle f, K_{\lambda}(\cdot) a - P_{i} K_{\lambda}(\cdot) a\right\rangle_{\mathcal{H}(\Gamma)}
\end{align*}
This means that
$$(L^{*} K_{\lambda}(\cdot) a)(z) = \frac{1}{\overline{b(\lambda)}} (K_{\lambda}(z) a - (P_{i} K_{\lambda}(\cdot) a)(z)).$$

A similar computation yields
$$(R K_{\lambda}(\cdot) a)(z) = b(z) (K_{\lambda}(z) a - (P_{-i} K_{\lambda}(\cdot) a)(z)).$$
Using Lemma \eqref{RLnn} we equate the two previous formulas for $(L^{*} K_{\lambda}(\cdot) a)(z)$ and $(R K_{\lambda}(\cdot) a)(z)$ and then work the algebra to get the result.
\end{proof}

\begin{Remark}\label{epl}
Let us  pause for a moment to remember, since this will be important for what follows,  that by the definition of a model $\Gamma$ and the formula for the reproducing kernel $K_{\lambda}(z) = \Gamma(z)^{*} \Gamma(\lambda)$, the operator $K_{\lambda}(z): \mathcal{K} \to \mathcal{K}$  is invertible whenever $\lambda, z \in \C_{+}$ or $\lambda, z\in \C_{-}$.
\end{Remark}

\begin{Lemma} \label{choice-phi}
If
$$\Phi(z) := K_{i}(z) K_{i}(i)^{-1/2}, \quad \Psi(z) := K_{-i}(z) K_{-i}(-i)^{-1/2},$$ then for all $a \in \mathcal{K}$,
$$(P_i K_{\lambda}(\cdot) a)(z) = \Phi(z) \Phi(\lambda)^{*} a, \quad (P_{-i} K_{\lambda}(\cdot) a)(z) = \Psi(z) \Psi(\lambda)^{*} a.$$
\end{Lemma}

\begin{proof}
For $z, \lambda \in \C \setminus \R$ and $a \in \mathcal{K}$ notice that
\begin{align*}
\Phi(z) \Phi(\lambda)^{*} a & = K_{i}(z) K_{i}(i)^{-1/2} K_{i}(i)^{-1/2} K_{\lambda}(i) a\\
& = K_{i}(z) K_{i}(i)^{-1} K_{\lambda}(i) a
\end{align*}
which, by \eqref{PH} belongs to $P_i \mathcal{H}(\Gamma)$. To finish the proof we need to show that
$$\langle K_{\lambda}(\cdot) a - \Phi(\cdot) \Phi(\lambda)^{*} a, \Phi(\cdot) \Phi(\lambda)^{*} b\rangle_{\mathcal{K}} = 0, \quad \forall b \in \mathcal{K}.$$ This can be routinely verified with the definition of the reproducing kernel as well as the fact that $K_{i}(i)$ is self-adjoint and invertible.
\end{proof}

\begin{Lemma} \label{L-inv}
For each $z \in \C_{+}$,  the operator $V(z): \mathcal{K} \to \mathcal{K}$ defined by $$V(z) := b(z) \Phi(z)^{-1} \Psi(z)$$ is  a strict contraction.
\end{Lemma}

\begin{proof}
It suffices to prove that $\|V(z)^{*}\|  < 1$ for each fixed $z$. To this end, note that the above technical lemmas show that
$$K_{\lambda}(\lambda) = \frac{1}{1 - |b(\lambda)|^2} (\Phi(\lambda) \Phi(\lambda)^{*} - |b(\lambda)|^2 \Psi(\lambda) \Psi(\lambda)^{*}), \quad \lambda \in \C_{+}.$$
Use this formula along with the estimate
$$\langle K_{\lambda}(\lambda) a, a\rangle_{\mathcal{K}} \geqslant \epsilon(\lambda) \|a\|^2,$$
(from Remark \ref{epl}) to get, after re-arranging some terms,
$$\epsilon(\lambda) (1 - |b(\lambda|^2) \|a\|^2 \leqslant \|\Phi(\lambda)^{*} a\|^2 - |b(\lambda)|^2 \|\Psi(\lambda)^{*} a\|^2.$$ Re-arrange the terms from the previous line to see that
$$|b(\lambda)|^2 \|\Psi(\lambda)^{*} a\|^2 \leqslant \|\Phi(\lambda)^{*} a\|^2 - \epsilon(\lambda) (1 - |b(\lambda|^2) \|a\|^2.$$ Insert
$$a = (\Phi(\lambda)^{*})^{-1} x$$ into the previous inequality, along with the definition fo $V$,  to obtain
$$\|V(\lambda)^{*} x\|^2 \leqslant \|x\|^2 - \epsilon(\lambda) (1 - |b(\lambda)|^2) \|(\Phi(\lambda)^{*})^{-1} x\|^2.$$
From here we see that $\|V(\lambda)^{*} x\| < \|x\|$ for all $x$. Suppose $\|x_{n}\|  = 1$ with $\|V(\lambda)^{*} x_{n}\| \to 1$. The previous inequality will show that
$$ \|(\Phi(\lambda)^{*})^{-1} x_n\| \to 0$$ which will contradict the fact that $(\Phi(\lambda)^{*})^{-1}$ is an invertible operator and hence must be bounded below. Thus $\|V(\lambda)^{*}\| < 1$.
\end{proof}

\begin{Remark}
Using the fact that $|b(z)| > 1$ when $z \in \C_{-}$, one can run the above proof again to show that $\|V(z)\| > 1$ when $z \in \C_{-}$. Note that $V(z)$ might have a pole when $z \in \C_{-}$. In fact if $V(z) = 0$ for some $z \in \C_{+}$ then $V$ will have a pole at $\overline{z}$ of the same order. See Remark \ref{zero-pole-V} below for more on this.
\end{Remark}

\begin{proof}[Proof of Theorem \ref{T-main-vector-kernel}] Since the statements of the theorem are contained in the above technical Lemmas, we just need to prove the formula for the kernel function. Indeed,
\begin{align*}
K_{\lambda}(z) & = \frac{\Phi(z) \Phi(\lambda)^{*} - b(z) \overline{b(\lambda)} \Psi(z) \Psi(\lambda)^{*}}{1 - b(z) \overline{b(\lambda)}}\\
& = \Phi(z) \left( \frac{\Phi(\lambda)^{*} - b(z) \overline{b(\lambda)} \Phi(z)^{-1} \Psi(z) \Psi(\lambda)^{*}}{1 - b(z) \overline{b(\lambda)}}\right)\\
& =  \Phi(z) \left( \frac{I - b(z) \overline{b(\lambda)} \Phi(z)^{-1} \Psi(z) \Psi(\lambda)^{*}(\Phi(\lambda)^{*})^{-1}}{1 - b(z) \overline{b(\lambda)}}\right) \Phi(\lambda)^{*}\\
& =  \Phi(z) \left( \frac{I - (b(z)  \Phi(z)^{-1} \Psi(z)) (\overline{b(\lambda)}\Psi(\lambda)^{*}(\Phi(\lambda)^{*})^{-1})}{1 - b(z) \overline{b(\lambda)}}\right) \Phi(\lambda)^{*}\\
& = \Phi(z) \left(\frac{I - V(z) V(\lambda)^{*}}{1 - b(z) \overline{b(\lambda)}}\right) \Phi(\lambda)^{*}.
\end{align*}
This completes the proof.
\end{proof}

The proof of Corollary \ref{Cor-main-Liv-ue} requires another technical lemma.

\begin{Lemma} \label{multiplier-matrx}
Suppose $M_{\Gamma_1} \cong M_{\Gamma_2}$ via the unitary operator $U: \mathcal{H}_1(\Gamma_1) \to \mathcal{H}_2(\Gamma_2)$. Then there is an analytic operator-valued function $W$ on $\C \setminus \R$ such that
$$(U f)(\lambda) = W(\lambda) f(\lambda), \quad f \in \mathcal{H}_1(\Gamma_1), \quad \lambda \in \C \setminus \R.$$
\end{Lemma}

\begin{proof}
For any $f \in \mathscr{D}(M_1)$ and $g \in \mathcal{H}_1(\Gamma_1)$ we have
$$\langle (M_{\Gamma_1} - \lambda I) f, g \rangle_{\mathcal{H}_1(\Gamma_1)} = \langle (M_{\Gamma_2} - \lambda) U f, U g \rangle_{\mathcal{H}_{2}(\Gamma_2)}. \label{ker-equation}$$
Thus $g \in \mbox{Rng}(M_{\Gamma_1} - \lambda I)^{\perp} \Leftrightarrow U g  \in \mbox{Rng}(M_{\Gamma_2} - \lambda I)^{\perp}$ and so $U$ maps $ \mbox{Rng}(M_{\Gamma_1} - \lambda I)^{\perp}$ onto $\mbox{Rng}(M_{\Gamma_2} - \lambda I)^{\perp}$.

By the identity
$$\bigvee \{K^{j}_{\lambda}(\cdot) a: a \in \mathcal{K}\} = \mbox{Rng}(M_{\Gamma_j} - \lambda I)^{\perp}, \quad j = 1, 2, $$
we see, for every $a \in \mathcal{K}$,  that $U K^{1}_{\lambda}(z) a \in \mbox{Rng}(M_2 - \lambda I)^{\perp}$ and so there exists an invertible operator $J(\lambda): \mathcal{K} \to \mathcal{K}$ with $U K^{1}_{\lambda}(\cdot) a = K^{2}_{\lambda}(\cdot) J(\lambda) a$. Then for any $f \in \mathcal{H}_1(\Gamma_1), a \in \mathcal{K},  \lambda \in \C \setminus \R$ we have
\begin{align*}
\langle f(\lambda), a\rangle_{\mathcal{K}} & = \langle f, K^{1}_{\lambda}(\cdot) a\rangle_{\mathcal{H}_1(\Gamma_1)}\\
& = \langle U f, U K^{1}_{\lambda}(\cdot) a\rangle_{\mathcal{H}_2(\Gamma_2)}\\
& = \langle U f, K^{2}_{\lambda}(\cdot) J(\lambda) a\rangle_{\mathcal{H}_{2}(\Gamma_2)}.
\end{align*}
Now let $a = J(\lambda)^{-1} b$ to get
\begin{align*}
\langle (U f)(\lambda), a\rangle_{\mathcal{K}} & = \langle f(\lambda), J(\lambda)^{-1}(b)\rangle_{\mathcal{K}}\\
& = \langle (J(\lambda)^{-1})^{*} f(\lambda), b\rangle_{\mathcal{K}}.
\end{align*}
If we now set $W(\lambda) = (J(\lambda)^{-1})^{*}$ then
$$(Uf)(\lambda) = W(\lambda) f(\lambda),$$
which completes our proof.
\end{proof}

\begin{Remark} \label{zero-pole-V}
Observe that the denominator in the above formula for $K_{\lambda}(z)$ in Theorem \ref{T-main-vector-kernel} vanishes when $z = \overline{\lambda}$ and thus the numerator must also vanish. This shows
$$V(z) V(\overline{z})^{*} = I, \quad z \in \C \setminus \R.$$
When $V$ has a zero or pole at $z$ the above formula must be interpreted in the usual way (poles cancel out the zeros).
\end{Remark}

\begin{proof}[Proof of Corollary \ref{Cor-main-Liv-ue}]
Suppose there are constant unitary matrices $Q$ and $R$ so that $V_1(z) = R V_{2}(z) Q$ for all $z \in \C_{+}$.  Using the fact that $V(z) V(\overline{z})^{*} = I$ for all $z \in \C \setminus \R$ we see that $V_1(z) = R V_{2}(z) Q$ for all $z \in \C \setminus \R$.

First let us relate the two kernel functions $K^{1}_{\lambda}(z)$ and $K^{2}_{\lambda}(z)$ for the spaces $\mathcal{H}_1(\Gamma_1)$ and $\mathcal{H}_2(\Gamma_2)$. Indeed,
\begin{align*}
K^{1}_{\lambda}(z) & = \Phi_{1}(z) \left( \frac{I - V_{1}(z) V_{1}(\lambda)^{*}}{1 - b(z) \overline{b(\lambda)}}\right) \Phi_{1}(\lambda)^{*}\\
& = \Phi_{1}(z) \left( \frac{I - R V_2(z) Q Q^{*} V_{2}(\lambda)^{*} R^{*}}{1 - b(z) \overline{b(\lambda)}}\right) \Phi_1(\lambda)^{*}\\
& = \Phi_{1}(z) \left( \frac{R R^{*} - R V_2(z) V_{2}(\lambda)^{*} R^{*}}{1 - b(z) \overline{b(\lambda)}}\right) \Phi_1(\lambda)^{*}\\
& = \Phi_1(z) R \left( \frac{I -  V_2(z) V_{2}(\lambda)^{*}}{1 - b(z) \overline{b(\lambda)}}\right)  R^{*} \Phi_{1}(z)^{*}\\
& = (\Phi_1(z) R) \left( \frac{I -  V_2(z) V_{2}(\lambda)^{*}}{1 - b(z) \overline{b(\lambda)}}\right) (\Phi_{1}(\lambda) R)^{*}.
\end{align*}
Work with the identity
$$b\Phi_{1}^{-1} \Psi_1 = R b \Phi_{2}^{-1} \Psi_{2} Q$$ to get
$$\Phi_{1} R = \Psi_{1} Q^{*} \Psi_{2}^{-1} \Phi_2.$$
Plug this into the above calculation for $K^{1}_{\lambda}(z)$ to see that
\begin{align*}
K^{1}_{\lambda}(z) & = (\Phi_{1}(z) R) \left( \frac{I -  V_2(z) V_{2}(\lambda)^{*}}{1 - b(z) \overline{b(\lambda)}}\right) (\Phi_{1}(\lambda) R)^{*}\\
& = (\Psi_{1}(z) Q^{*} \Psi_{2}(z)^{-1} \Phi_2(z))   \left( \frac{I -  V_2(z) V_{2}(\lambda)^{*}}{1 - b(z) \overline{b(\lambda)}}\right) (\Psi_{1}(\lambda) Q^{*} \Psi_{2}(\lambda)^{-1} \Phi_2(\lambda))^{*}\\
& = (\Psi_{1}(z) Q^{*} \Psi_{2}^{-1}(z)) K^{2}_{\lambda}(z) (\Psi_{1}(\lambda) Q^{*} \Psi_{2}^{-1}(\lambda))^{*}\\
& = G(z) K^{2}_{\lambda}(z) G(\lambda)^{*},
\end{align*}
where
$$G(z) = \Psi_{1}(z) Q^{*} \Psi_{2}^{-1}(z).$$

Notice that
\begin{equation}
\bigvee\{K^{j}_{\lambda}(\cdot) a: a \in \mathcal{K}\} = \mbox{Rng}(M_{\Gamma_j} - \lambda I)^{\perp} = \mbox{Ker} (M_{\Gamma_j}^{*} - \overline{\lambda} I)
\end{equation}
 and
$$\bigvee\{\mbox{Ker} (M_{\Gamma_j}^{*} - \overline{\lambda} I) : \lambda \in \C \setminus \R\} = \mathcal{H}_j(\Gamma_j).$$
Thus we can define the operator
$U: \mathcal{H}_{1}(\Gamma_1) \to \mathcal{H}_2(\Gamma_2)$ first as
$$U K^{1}_{\lambda}(\cdot) a = K^{2}_{\lambda}(\cdot) G(\lambda)^{*} a, \quad a \in \mathcal{K}, \lambda \in \C \setminus \R,$$ and then extend linearly.
We have the following computation
\begin{align*}
\langle U K^{1}_{\lambda}(\cdot) a, U K^{1}_{\eta}(\cdot) b\rangle_{\mathcal{H}_2(\Gamma_2)} & =
\langle K^{2}_{\lambda}(\cdot) G(\lambda)^{*} a, K^{2}_{\eta}(\cdot) G(\eta)^{*} b\rangle_{\mathcal{H}_2(\Gamma_2)}\\
& = \langle K^{2}_{\lambda}(\eta) G(\lambda)^{*} a, G(\eta)^{*} b\rangle_{\mathcal{K}}\\
& = \langle G(\eta) K^{2}_{\lambda}(\eta) G(\lambda)^{*} a, b\rangle_{\mathcal{K}}\\
& = \langle K^{1}_{\lambda}(\eta) a, b\rangle_{\mathcal{K}}\\
& = \langle K^{1}_{\lambda}(\cdot) a, K^{1}_{\eta}(\cdot) b\rangle_{\mathcal{K}}.
\end{align*}
This says that $U$ is a unitary operator. For $f \in \mathcal{H}_{2}(\Gamma_2)$ we have
\begin{align*}
\langle (U^{*} f)(\lambda), a\rangle_{\mathcal{K}} & = \langle U^{*} f, K^{1}_{\lambda}(\cdot) a\rangle_{\mathcal{H}_1(\Gamma_1)}\\
& = \langle f, U K^{1}_{\lambda}(\cdot) a \rangle_{\mathcal{H}_{2}(\Gamma_2)}\\
& = \langle f(\lambda), G(\lambda)^{*} a\rangle_{\mathcal{K}}\\
& = \langle G(\lambda) f(\lambda), a\rangle_{\mathcal{K}}.
\end{align*}
Thus $(U^{*} f)(\lambda) = G(\lambda) f(\lambda)$ and $M_{\Gamma_1}$ is unitarily equivalent to $M_{\Gamma_2}$ via the unitary $U$.
We have just shown that $V_1 = R V_2 Q$ implies $M_{\Gamma_1} \cong M_{\Gamma_2}$.

So now suppose that $M_{\Gamma_1} \cong M_{\Gamma_2}$ via a unitary operator $U: \mathcal{H}_{1}(\Gamma_1) \to \mathcal{H}_{2}(\Gamma_2)$. Then by Lemma \ref{multiplier-matrx} there is an analytic operator-valued function $W$ so that $U f = W f$. Furthermore since $M_{W}$ (multiplication by $W$) takes, for each fixed $\lambda \in \C \setminus \R$,  $\mbox{Rng}(M_{\Gamma_1} - \lambda I)^{\perp}$ onto $\mbox{Rng}(M_{\Gamma_2} - \lambda I)^{\perp}$, we get that for each $a \in \mathcal{K}$,
$$M_{W} K^{1}_{\lambda}(\cdot) K^{1}_{\lambda}(\lambda)^{-1/2} a = K^{2}_{\lambda}(\cdot) K^{2}_{\lambda}(\lambda)^{-1/2} R(\lambda)a$$ for some invertible linear operator $R(\lambda): \mathcal{K} \to \mathcal{K}$.
Observe that for any $a, b \in \mathcal{K}$,
\begin{align*}
& \langle M_{W} K^{1}_{\lambda}(\cdot) K^{1}_{\lambda}(\lambda)^{-1/2} a, M_{W} K^{1}_{\lambda}(\cdot) K^{1}_{\lambda}(\lambda) b\rangle_{\mathcal{H}_{2}(\Gamma_2)}\\
 & = \langle K^{1}_{\lambda}(\cdot) K^{1}_{\lambda}(\lambda)^{-1/2} a, K^{1}_{\lambda}(\cdot) K^{1}_{\lambda}(\lambda) b\rangle_{\mathcal{H}_{1}(\Gamma_1)}\\
 & = \langle K^{1}_{\lambda}(\lambda) K^{1}_{\lambda}(\lambda)^{-1/2} a, K^{1}_{\lambda}(\lambda)^{-1/2} b\rangle_{\mathcal{K}}\\
 & = \langle a, b \rangle_{\mathcal{K}}.
\end{align*}
On the other hand,
\begin{align*}
& \langle M_{W} K^{1}_{\lambda}(\cdot) K^{1}_{\lambda}(\lambda)^{-1/2} a, M_{W} K^{1}_{\lambda}(\cdot) K^{1}_{\lambda}(\lambda) b\rangle_{\mathcal{H}_{2}(\Gamma_2)}\\
& = \langle K^{2}_{\lambda}(\cdot) K^{2}_{\lambda}(\lambda)^{-1/2} R(\lambda)a, K^{2}_{\lambda}(\cdot) K^{2}_{\lambda}(\lambda)^{-1/2} R(\lambda)b \rangle_{\mathcal{H}_{2}(\Gamma_2)}\\
& = \langle R(\lambda) a, R(\lambda) b \rangle_{\mathcal{K}}.
\end{align*}
This implies that $R(\lambda): \mathcal{K} \to \mathcal{K}$ is unitary for each $\lambda \in \C \setminus \R$.
We leave it to the reader to check that the following two identities
$$W(z) K^{1}_{\lambda}(z) K^{1}_{i}(i)^{-1/2}  = K^{2}_{i}(z) K^{2}_{i}(i)^{-1/2} R(i)$$
$$W(z) K^{1}_{\lambda}(z) K^{1}_{-i}(-i)^{-1/2}  = K^{2}_{-i}(z) K^{2}_{-i}(-i)^{-1/2} R(-i)$$
yield
$$R(i) V_1(z) = V_2(z) R(-i)$$ which completes the proof.
\end{proof}

We can now prove that the invariant $V$ from Lemma \ref{L-inv} and Livsic's $w_{T}$ from \eqref{intro-Livsic} are indeed equivalent.

\begin{proof}[Proof of Corollary \ref{VisW}]
Let $\Gamma$ be a model for $T$, and consider the model space $\mathcal{H}(\Gamma)$.
Hence there is an $n-$dimensional Hilbert space $\mathcal{K}$ such that $\Gamma (z) : \mathcal{K} \rightarrow \mbox{Rng}(T- z I ) ^\perp $ is bounded
and invertible for each $z \in \C \setminus \R$. Suppose that $\{ e_j \}_{j = 1}^{n}$ is a fixed orthonormal basis for $\mathcal{K}$.

From our earlier work, $T$ is unitarily equivalent to $M := M_{\Gamma}$ which acts (densely) as multiplication by $z$ on $\mathcal{H} (\Gamma  )$, and so, without loss of
generality, we assume that $T = M$ and $\mathcal{H} = \mathcal{H} (\Gamma )$. Then the Lisvic characteristic function for $T$ is
$$ w_T (z) := b(z) B(z) ^{-1} A(z),$$ where
$$B(z) = \left[ \langle (M' -i I)(M' -z I)^{-1} u_j, u_k \rangle \right] ,$$
$$A(z) = \left[ \langle (M' + i I)(M' -z I)^{-1} u_j, u_k \rangle \right] ,$$
$\{ u_k \}_{k = 1}^{n}$ is any orthonormal basis for $\mbox{Ker}(M^* - i I)$ and $M'$ is some fixed (canonical) self-adjoint extension of $M$.  Now since $\{ e_j \}_{j = 1}^{n}$ is orthonormal, we see that
$$u_j := K _{-i} (\cdot) K_{-i} (-i) ^{-1/2} e_j, \quad 1 \leqslant j \leqslant n,$$ forms an orthonormal basis for $\mbox{Ker}(M^* -i I)$. Indeed,
\begin{align*}
\langle u_j, u_k \rangle_{\mathcal{H}} & = \langle K_{-i}(\cdot) K_{-i}(-i)^{-1/2} e_j, K_{-i}(\cdot) K_{-i}(-i)^{-1/2} e_k\rangle_{\mathcal{H}}\\
& = \langle K_{-i}(-i) K_{-i}(-i)^{-1/2} e_j, K_{-i}(-i)^{-1/2} e_k \rangle_{\mathcal{K}}\\
& = \langle e_j, e_k \rangle_{\mathcal{K}}.
\end{align*}

 Thus we can assume, with at most creating an equivalent Livsic characteristic function,  the $u_j$ have this form. From our Krein trick in \eqref{K-form} we also know, for any $z \in \C \setminus \R$,  that
$$(M' -i I ) (M' -z I) ^{-1} : \mbox{Ker}(M^* -i I) \rightarrow \mbox{Ker}(M^* -z I)$$ is bounded and invertible, so that we can find
a bounded invertible operator $V_z : \mathcal{K} \rightarrow \mathcal{K} $ such that
$$(M' -i I ) (M' - \ov{z} I) ^{-1} u_j = K _{z} (\cdot ) V_z e_j.$$  Here we are using the fact that $\mbox{Ker}(M^* -\ov{z} I )$ is spanned
by the vectors $$K _z (\cdot) e_j, \quad 1 \leqslant j \leqslant n.$$
 Actually they form a Riesz basis which ensures that $V_z$ is bounded and invertible when $n=\infty$.

We can now compute $A(z)$ as
\begin{align*}
A(z)   = & \left[ \langle (M' + i I)(M' -z I)^{-1} u_j, u_k \rangle \right] \nonumber \\
 = &  \left[ \langle u_j, (M' - i I)(M' -\ov{z} I )^{-1} u_k \rangle \right] \nonumber \\
 = &  \left[ \langle K_{-i} (\cdot) K_{-i} (-i) ^{-1/2} e_j , K_z (\cdot) V_{z} e_k \rangle \right] ,
\end{align*}
and it follows that
$$A(z) = V_{z} ^* K_{-i} (z) K _{-i} (-i) ^{-1/2} = V_z ^* \Psi (z) .$$
Similarly,
\begin{align*}
B(z) & =  \left[ \langle (M' - i I)(M' -z I )^{-1} u_j, u_k \rangle \right] \nonumber \\
& = \left[ \langle  (M' -iI) (M' +iI) ^{-1}  u_j,  (M' -i I ) (M' - \ov{z} I ) ^{-1}  u_k \rangle \right]  \nonumber \\
&  =  \left[ \langle K_i (\cdot) K_i (i) ^{-1/2} U e_j ,  K _z (\cdot ) V_z e_k \rangle \right].
\end{align*}
Here we have that $U : \mathcal{K} \rightarrow \mathcal{K}$ is some fixed unitary operator. The existence of $U$ follows from the facts that
$(M' -i) (M' +i) ^{-1}$ is unitary and that for any orthonormal basis $\{ b_j \}_{j = 1}^{n}$ of $\mathcal{K}$, $K_i (\cdot )K_i (i) ^{-1/2} b_j$ is an orthonormal
basis of $\mbox{Ker} (M^* +iI)$.
This shows that $B(z) = V_z ^* K_i (z) K _i (i) ^{-1/2} U = V_z ^*  \Phi (z) U$.
Hence
$$ w_T (z) = b(z) B(z) ^{-1} A(z) = U^* b(z) \Phi (z) ^{-1}  ( V_z ^* ) ^{-1} V_z ^* \Psi (z) = U^* V_T (z) .$$
Since $U$ is a constant unitary matrix we conclude that $V_T$ and $w_T$ are equivalent.

\end{proof}

\section{Computing the characteristic function}

Let us compute the characteristic functions $V$ for the examples mentioned earlier. After discussing Herglotz spaces we will use these computations to make come interesting connections to these operators to vector-valued deBranges-Rovnyak spaces.

\subsection{Differentiaton}
Let us return to the differentiation example $Tf = i f'$ on $L^2[-\pi, \pi]$ with domain $\{f \in L^2[-\pi, \pi]: f(-\pi) = f(\pi) = 0\}$ from Example \ref{diff-ex}. We saw that the corresponding Hilbert space of analytic functions on $\C \setminus \R$ was the Paley-Wiener (type) space with reproducing kernel
$$K_{\lambda}(z) = 2 \frac{\sin \pi(z - \overline{\lambda})}{z - \overline{\lambda}}.$$
A computation will show that
$$K_{i}(i) = K_{-i}(-i) = \sinh \pi$$ and thus
the Livsic function $V(z)$ is
$$V(z) = \left(\frac{z - i}{z + i}\right) \frac{\sin \pi (z - i)}{\sin\pi (z + i)}, \quad z \in \C_{+}.$$
Notice how $|V(z)| < 1$ on $\C_{+}$ with zeros $\{i + \pi n: n \in \mathbb{Z}\}$ and $|V(z)| > 1$ on $\C_{-}$ with poles $\{-i + \pi n: n \in \mathbb{Z}\}$. A computation will show that $|V(x)| = 1$ for $x \in \R$ and so $V$ is inner on $\C_{+}$. This will be important later on.

\subsection{Double differentiation}
For the operator  $T f = -f''$ discussed in Example \ref{double-diff}, the kernel function $K_{\lambda}(z)$ was computed to be
$$K_{\lambda}(z) = \frac{-i}{\epsilon_{z} \sqrt{z} - \epsilon_{\lambda} (\overline{\lambda} ) ^{\frac{1}{2}}},$$
where $\epsilon_{\lambda} =  1$ if $\Im \lambda > 0$ and $\epsilon_{\lambda} = -1$ if $\Im \lambda < 0$.
In this example,
$$K_{i}(i) = K_{-i}(-i) = \frac{1}{\sqrt{2}}$$ and so
$$V(z) =\left( \frac{z - i}{z + i}\right) \frac{\sqrt{z} - \frac{1-i}{\sqrt{2}}}{\sqrt{z} + \frac{1+i}{\sqrt{2}}}, \quad z \in \C_{+}.$$
One needs to be careful when computing $V(z)$ for $\Im z < 0$ since $\epsilon_z = -1$ and so $\sqrt{z}$ changes to a $-\sqrt{z}$ in the above formula for $V(z)$.
A computation will show that $|V(x)| = 1$ for $x < 0$ and so, although $V$ is not inner on $\C_{+}$, it is an extreme function for the unit ball of $H^{\infty}(\C_{+})$ -- which will become important later when we discuss deBranges-Rovnyak spaces.

\subsection{Sturm-Liouville operators}
This example is a continuation of Example \ref{SL-ex}. We will assume that $I$ is a closed finite interval such that $1/p , q \in L^1 (I)$. Recall that in this case the operators $H(p,q,I)$ which act as
\[ H(p,q,I) f = -(pf')' + qf, \] for all $f \in \mathscr{D} (H (p,q,I)) \subset L^2 (I)$ are closed simple symmetric densely defined operators with indices $(2,2)$.

In this case where $I$ is finite and $1/p, q \in L^1 (I)$, $H (p,q,I)$ is called a regular second-order Sturm-Liouville differential operator,  and it is known that $H(p,q,I) -xI$ is bounded below for
any $x \in \mathbb{R}$, so that every $x \in \mathbb{R}$ is a regular value of $H(p,q,I)$. Recall that any symmetric operator with this property is called regular.

Here is a brief sketch of a proof that $H(p,q,I)$ is regular: It can be proven that the domain of $H(p,q,I)$ is the set of all $f \in \mathscr{D} ( H(p,q,I) ^* )$ such that
both $f(a) = 0 = f(b)$ and $p(a) f'(a) = 0 = p(b) f'(b)$ \cite[Lemma 1, Section 17.3]{Naimark}. Hence if $x \in \mathbb{R}$ was an eigenvalue of $H(p,q,I)$ with corresponding eigenfunction $f \in \mathscr{D} (H(p,q,I))$, $f$ would be a solution to the ordinary differential
equation:  \[ -(pf')' +qf  = x f, \] which obeys the boundary conditions $f(a) = 0 $ and $p(a) f'(a) =0$. The existence-uniqueness theorem for ordinary differential equations \cite[Theorem 2, Section 16.2]{Naimark} would then imply that $f=0$. This contradiction proves that $H(p,q,I)$ has no eigenvalues. Now by \cite[Theorem 1, Section 19.2]{Naimark}, the resolvent $ (H - z I) ^{-1}$, where $z \in \mathbb{C} \setminus \mathbb{R}$ and $H$ is any fixed
self-adjoint extension of $H(p,q,I)$, is a compact Hilbert-Schmidt integral operator. It follows that the spectrum of any self-adjoint extension $H$ of $H(p,q,I)$ is a discrete sequence of eigenvalues with no finite accumulation point, and $H$ has no finite essential spectrum. If for some $x \in \mathbb{R}$, $H(p,q,I) -x I$ was not bounded below, then since $x$ cannot be an eigenvalue, it would have to belong to the essential spectrum of $H(p,q,I)$. It follows from \cite[Theorem 1, Section 83]{A-G} that $x$ would have to belong to the essential spectrum of every self-adjoint extension $H$ of $H(p,q,I)$, and this contradicts the fact that the essential spectrum of any such self-adjoint extension is empty. Note here that the Cayley transforms $b(H)$ and $b(H(p,q,I))$ of $H$ and $H(p,q,I)$ differ by a finite rank perturbation, so this also follows from the fact that any two bounded operators which differ by a compact perturbation have the same essential spectrum. In conclusion $H(p,q,I) -xI $ is bounded below for any $z \in \mathbb{C}$, and $H(p,q,I)$ is regular. Also note that any regular symmetric operator $T$ must also be simple, as if $T$ had a self-adjoint restriction $T_0$, then $T_0$ would have spectrum so that $T_0 -xI$ and hence $T-xI $ would not be bounded below for some $x \in \mathbb{R}$. This provides another proof that $H(p,q,I)$ is simple in this case.

Let $V_I$ denote the characteristic function of $H(p,q,I)$. By a result of Livsic, \cite[Theorem 4]{Livsic-2}, since every $x \in \mathbb{R}$ is a regular point of $H(p,q, I)$, it follows that $V_I$ is a $2 \times 2$ matrix-valued inner function which has an analytic extension to a neighborhood of $\mathbb{R}$. To actually compute this inner characteristic function $V_I$, for any $z \in \mathbb{C}$, let $u_z , v_z$ be the entire $L^2 (I) -$valued functions spanning $\mbox{Ker} (H(p,q,I)^* -z I ) $ discussed in Example \ref{SL-ex}.
As in Example \ref{SL-ex}, if we define
$$\gamma _1 (z) = u_z, \quad \gamma _2 (z) = v_z$$
and
$$\Gamma _I (z) = \gamma_1 (\overline{z} ) \otimes e_1 + \gamma _2 (\overline{z} ) \otimes e_2,$$
then $\Gamma: \C \setminus \R \rightarrow \mbox{Ker} (H (p,q,I) ^* -\overline{z} I)$
is a model for $H (p,q,I)$ and $\mathcal{H} (\Gamma _I)$ has reproducing kernel: \[ K ^I _\lambda (z) = \left( \begin{array}{cc}  \int _I  u_{\overline{\lambda}} (x)  \overline{ u _{\overline{z}} (x) } dx    &
\int _I u _{\overline{\lambda}} (x)  \overline{ v_{\overline{z}} (x)} dx  \\  \int _I  v_{\overline{\lambda}} (x) \overline{ u _{\overline{z}} (x) } dx  &
\int _I  v _{\overline{\lambda}} (x)  , \overline{ v_{\overline{z}}  (x) } dx   \end{array}  \right). \] From this one can compute the characteristic function $V_I$ as
$$V_I (z) = b(z) \Phi ^I (z) ^{-1} \Psi ^I  (z)$$
where
$$b(z) = \frac{z-i}{z+i},$$
$$\Phi ^I  (z) = K_i ^I (z) K_i ^I (i) ^{-1/2},$$
and
$$\Psi ^I (z) = K_{-i} ^I (z) K_{-i} ^I (-i) ^{-1/2}.$$
Note that both $\Phi ^I (z)$ and $\Psi ^I (z)$ are entire matrix functions of $z \in \mathbb{C}$.

Now consider a larger interval $J \supset I$, and repeat the above arguments for the operator $H (p,q , J )$ acting on its dense domain in $L^2 (J)$. Note that $H(p,q,J) \supset H(p,q,I)$, \emph{i.e.}
$$\mathscr{D} (H(p,q,J) ) \supset \mathscr{D} (H(p,q,I)), \quad H(p,q,J) |\mathscr{D} (H(p,q,I)) = H(p,q,I).$$
 Observe that if $U_J : L^2 (J) \rightarrow \mathcal{H} (\Gamma _J ) $ is the isometry defined by
$$ (U_J f ) (z) = \Gamma _J (z) ^* f = \left( \langle f , u_{\overline{z}} \rangle _J , \langle f , v_{\overline{z}} \rangle _J \right),$$ that $U_J |L^2 (I) = U_I$ where $U_I$ is the corresponding
isometry of $L^2 (I) $ onto $\mathcal{H} (\Gamma _I)$ which takes $H(p,q,I)$ onto $M _I = U_I H(p,q,I) U_I ^*$, the symmetric operator of multiplication by $z$ in $\mathcal{H} (\Gamma _I )$. This shows
that $\mathcal{H} (\Gamma _I )$ is a closed subspace of $\mathcal{H} (\Gamma _J )$, and that if $M _J = U_J H(p,q, J) U_J ^*$
is the corresponding operator of multiplication by $z$ in $\mathcal{H} (\Gamma _J )$ then $M_I \subset M _J$. Since the characteristic functions $V_I$ and $V_J$ are inner, it will further follow from Theorem \ref{multiplier} of the last section of this paper
that multiplication by $\frac{\Phi _I }{\Phi _J}$, where $\Phi _I (z) = K ^I _i (z) K^I _i (i) ^{-1/2}$, is an isometric multiplier from the model subspace $\mathscr{K} _I := H^2 _{\mathbb{C} ^2} \ominus V_I H^2 _{\mathbb{C} ^2}$
into $\mathscr{K} _J$. These observations seem to be connected to the results of \cite{Remling}, and although we will not pursue this further here, it would be interesting to investigate this in a future paper.

\subsection{Multiplication by the independent variable}
For the example of $M_{\mu}$, the restriction of $M^{\mu}$, multiplication by the independent variable on $L^2(\mu)$,  to
$$\left\{f \in L^2(\mu): x f \in L^2(\mu), \int f d \mu = 0\right\},$$ recall that the reproducing kernel for the corresponding Hilbert space of analytic functions on $\C \setminus \R$ (the Cauchy transforms of $L^2(\mu)$ functions) is
$$K_{\lambda}(z) = \int \frac{1}{(x - z)(x - \overline{\lambda})} d \mu(x).$$
Notice that
$$K_{i}(i) = K_{-i}(-i) = \int \frac{1}{1 + x^2} d \mu(x)$$ and so
the Livsic characteristic function is
$$V(z) = \frac{z - i}{z + i} \frac{\int \frac{d \mu(t)}{(t - i)(t - z)}}{\int \frac{d \mu(t)}{(t + i)(t - z)}}, \quad z \in \C_{+}.$$
One can use the Poisson integral theory to show that $V$ is inner on $\C_{+}$ if and only if $\mu$ is singular with respect to Lebesgue measure on $\R$.
If $\mu$ has no support on an interval $I \subset \R$, then, since $|V(x + i y)| < 1$ for $x \in I, y > 0$ while $|V(x + i y)| > 1$ for $x \in I, y < 0$, and $V$ has an obvious analytic continuation across $I$, we see that $|V(x)| = 1$ on $I$. Though $V$, in this case where the support of $\mu$ omits an interval,  may not be inner (unless $\mu$ is singular with respect to Lebesgue measure), it is an extreme function (see the definition of extreme functions in the last section).

\subsection{Toeplitz operators}
Recall the Toeplitz operator example $T_g, g \in N^{+}_{\R}$ from Example \ref{Toeplitz-ex}. Note that $T_{g} \in \mathcal{S}_{1}(H^2)$ precisely when
$$g = i \frac{p + q}{p - q},$$
where $p, q$ are order one Blaschke products such that $p - q$ is outer. One can easily check that
$$p(z) = z, \quad q(z) = \frac{z- a}{1 - a z}, \quad 0 < a < 1,$$
work.

 One can show that when $a = 1/2$, $g$ maps $\mathbb{D}$ onto $\C \setminus ((-\infty, \sqrt{3}] \cup [\sqrt{3}, \infty))$ and
$$g^{-1}(z) = \frac{\sqrt{z^2 - 3}-2 i}{z-i}.$$
As worked out in Example \ref{Toeplitz-ex} we saw that
\begin{equation} \label{K-Toep}
K_{\lambda}(z) = \frac{1}{1 - \overline{g^{-1}(\lambda)} g^{-1}(z)}
\end{equation}
and a computation will show that the corresponding Livsic function $V$ is
$$V(z) = -\frac{\sqrt{z^2 - 3}-2 z}{\sqrt{3} (z+i)}, \quad z \in \C_{+}.$$
Another computation will show that that $|V(x)| = 1$ on $[-\sqrt{3}, \sqrt{3}]$ and so $V$ is extreme.

Let show how the formula \eqref{K-Toep} can be used to prove a theorem about unitary equivalence of symmetric Toeplitz operators.

\begin{Theorem} \label{Toeplitz-1-1}
Let  $g, h \in N^{+}_{\R}$ be such that  $T_{g}, T_{h} \in \mathcal{S}_{1}(H^2)$. Then $T_{g} \cong T_{h}$ if and only  if $g = h (w)$ where $w$ is a disk automorphism.
\end{Theorem}

\begin{proof}
If $g = h \circ w$ then the unitary operator $U: H^2 \to H^2$, $U f = \sqrt{w'} (f \circ w)$ satisfies $U T_{h} = T_{g} U$ and so $T_{g} \cong T_{h}$.

For the other direction, assume $T_{g} \cong T_{h}$. By composing with disk automorphisms, which will not change the unitary equivalence of $T_g$ and $T_h$,  we can assume that $g(0) = h(0) = -i$. Recall that the kernels $K^{g}$ and $K^{h}$ for the associated spaces corresponding to $T_{g}$ and $T_{h}$ are given by
$$K^{g}_{\lambda}(z) = \frac{1}{1 - \overline{g^{-1}(\lambda)} g^{-1}(z)}, \quad K^{h}_{\lambda}(z) = \frac{1}{1 - \overline{h^{-1}(\lambda)} h^{-1}(z)}.$$
Since $T_{g} \cong T_{h}$ we have
$$\frac{K^{g}_{-i}(z)/\|\cdot\|}{K^{g}_{i}(z)/\|\cdot\|} = \zeta \frac{K^{h}_{-i}(z)/\|\cdot\|}{K^{h}_{i}(z)/\|\cdot\|}$$
for some $|\zeta| = 1$. This reduces to the identity
$$\frac{1 - g^{-1}(z) \overline{g^{-1}(i)}}{\sqrt{1 - |g^{-1}(i)|^2}} = \zeta \frac{1 - h^{-1}(z) \overline{h^{-1}(i)}}{\sqrt{1 - |h^{-1}(i)|^2}}.$$
Plug in $z = i$ into the above identity to show that $\zeta = 1$ and $|g^{-1}(i)| = |h^{-1}(i)|$. A little algebra will now show that
$$g^{-1}(z) = \frac{\overline{h^{-1}(i)}}{\overline{g^{-1}(i)}} h^{-1}(z)$$ and moreover,
$$\frac{\overline{h^{-1}(i)}}{\overline{g^{-1}(i)}} = \alpha$$ is unimodular. Letting $z = h(t)$ for some $|t| < 1$ we see that
$$g^{-1}(h(t)) = \alpha t$$ and so
$g^{-1} \circ h$ is a disk automorphism.
\end{proof}

\begin{Question}
For $g, h \in N^{+}_{\R}$ with $T_{g}, T_{h} \in \mathcal{S}_{n}(H^2)$, when is $T_g \cong T_h$?
\end{Question}

For the general case, the answer is unknown but we can make a few general remarks.

\begin{Proposition}
If $g \in N_{\R}^{+}$ and $T_{g} \in \mathcal{S}_n(H^2)$, $n < \infty$, then the point spectrum $\sigma_{p}(T_{g}^{*})$ of $T_{g}^{*}$ satisfies
$$\sigma_{p}(T_{g}^{*}) = \left\{\overline{g(z)}: z \in \mathbb{D}\right\}.$$
\end{Proposition}

\begin{proof}
Since $T_{g}^{*} k_{\lambda} = \overline{g(\lambda)} k_{\lambda}$ we have
$$\left\{\overline{g(z)}: z \in \mathbb{D}\right\} \subset \sigma_{p}(T_{g}^{*}).$$ For the other direction, suppose
$T_{g}^{*} f = \eta f$ for some $f \in \mathscr{D}(T_{g}^{*}) \setminus \{0\}$. Then
$f \perp \mbox{Rng}(T_{g} - \overline{\eta}I )$ or equivalently
$$\langle f, (T_{g} - \overline{\eta}I) h\rangle = 0, \quad \forall h \in \mathscr{D}(T_{g}).$$
But writing $g = b/a$ in the Sarason decomposition from \eqref{gisba}, we see that $\mathscr{D}(T_{g})  = a H^2$ and so
$$\langle f, (b/a - \overline{\eta}) a w\rangle = 0, \quad \forall w \in H^2,$$
which implies
$$\langle f, (b - \overline{\eta} a) w \rangle = 0, \quad \forall w \in H^2.$$
However, since $f \not \equiv 0$, it must be the case that
$b - \overline{\eta} a$ has an inner factor. But since $a$ and $b$ are rational functions (Sarason proves that if $g$ is rational then so are $a$ and $b$) we see that this inner factor is a finite Blaschke product and so $b - \overline{\eta} a$ must vanish for some $z \in \mathbb{D}$, i.e., $g(z) = \overline{\eta}$. Thus we have the inclusion
$$\sigma_{p}(T_{g}^{*}) \subset \left\{\overline{g(z)}: z \in \mathbb{D}\right\},$$
which completes the proof.
\end{proof}

\begin{Corollary}
Suppose $g_1, g_2 \in N_{\R}^{+}$ with $T_{g_1}, T_{g_2} \in \mathcal{S}_{n}(H^2), n \in \mathbb{N}$. If $T_{g_1} \cong T_{g_2}$, then $g_1(\mathbb{D}) = g_2(\mathbb{D})$.
\end{Corollary}

\begin{proof}
If $U T_{g_1} = T_{g_2} U$, where $U: H^2 \to H^2$ is unitary with $U \mathscr{D}(T_{g_1}) = \mathscr{D}(T_{g_2})$, then $U(T_{g_1} - \lambda I) = (T_{g_2} - \lambda I)$ for all $\lambda \in \C$. So if $g \in H^2$ and $f \in \mathscr{D}(T_{g_1})$ with
$$\langle (T_{g_1}  - \lambda I) f, g \rangle = 0,$$
then
$$\langle (T_{g_2} - \lambda I) U f, U g \rangle = 0.$$
This means that
$$g \in \mbox{Rng}(T_{g_1} - \lambda I)^{\perp} \Leftrightarrow U g \in \mbox{Rng}(T_{g_2} - \lambda I)^{\perp}$$ and so
$$\mbox{Ker}(T_{g_1}^{*} - \overline{\lambda} I) \not = \{0\} \Leftrightarrow \mbox{Ker}(T_{g_2}^{*} - \overline{\lambda} I) \not = \{0\}.$$
This means that $\sigma_{p}(T_{g_1}^{*}) = \sigma_{p}(T_{g_2}^{*})$. By the previous proposition we conclude that $g_1(\mathbb{D}) = g_2(\mathbb{D})$.
\end{proof}

\begin{Remark}
Notice how the previous corollary gives us a proof of Theorem \ref{Toeplitz-1-1} which comes from general principles and does not involve the Livsic characteristic function.
\end{Remark}

Suppose $T_{g} \in \mathcal{S}_n(H^2)$ and we want to compute the Livsic characteristic function. In this case
$$\mbox{Ker}(T^{*}_{g} - \overline{\lambda} I) = \bigvee\{k_{z_{j}(\lambda)}: 1 \leqslant j \leqslant n\},$$ where
$z_{1}(\lambda), \cdots, z_{n}(\lambda)$ are the solutions to $g(z) = \lambda$.

If, and this is not always the case, the $g$ is such that the $z_j(\lambda)$ can be  chosen so $\lambda \mapsto z_{j}(\lambda)$ is analytic on $\C \setminus \R$. Then we can use our model discussed earlier and define
$$\gamma(\lambda) = (k_{z_{1}(\lambda)}, \cdots, k_{z_{n}(\lambda)}).$$
The Hilbert space $\mathcal{H}(\Gamma)$ is then
$$\{(f(z_{1}(\lambda)), \cdots, f(z_{n}(\lambda))): f \in H^2\}$$ with inner product
$$\langle (f_1(z_{1}(\lambda)), \cdots, f_1(z_{n}(\lambda))), (f_2(z_{1}(\lambda)), \cdots, f_2(z_{n}(\lambda))\rangle_{\mathcal{H}(\Gamma)} = \langle f_1, f_2 \rangle_{H^2}.$$
By our earlier discussion, the reproducing kernel is
$$K_{\lambda}(z) = [k_{z_i(\lambda)}(z_{j}(z))]_{1 \leqslant i, j \leqslant n}$$
and the Livsic characteristic function can be computed from here.

Can the above situation actually happen? Yes. Consider the case where $g \in N^{+}_{\R}$ and $T_g \in \mathcal{S}_1(H^2)$ (and consequently $g$ will be univalent). Then $T_{g^2} \in \mathcal{S}_{2}(H^2)$ and to solve $g(z)^2 = \lambda$ we must solve
$g(z) = \pm \sqrt{\lambda}$, which, at the end of the day (and since $g$ is invertible) will yield $z_1(\lambda)$ and $z_2(\lambda)$ analytic on $\C \setminus \R$. Note how the kernel function in this case was computed earlier.

So what does this all mean? From our version of Livsic's theorem we know that $T_{g_1}$ is unitarily equivalent to $T_{g_2}$ if and only if $V_1(\lambda) = R V_{2}(\lambda) Q$, where $V_1$ and $V_2$ are created from the above expression for $K_{\lambda}(z)$. Is it possible to translate this into a more workable condition -- as in the $(1, 1)$ case where $T_{g_1}$ is unitarily to $T_{g_2}$ if and only if $g_1 = g_2 \circ h$ where $h$ is a disk automorphism.

The more likely situation is when the functions $\lambda \mapsto z_j(\lambda)$ are only locally analytic -- to avoid where $g' = 0$. In this case, by Grauert's construction (or really the Krein construction) we have
$$\gamma(\lambda) = (\gamma_{1}(\lambda), \cdots, \gamma_n(\lambda)),$$
where
$$\gamma_{i}(\lambda) = \sum_{j = 1}^{n} \overline{a_{i, j}(\lambda)} k_{z_{j}(\lambda)}.$$
The functions $\lambda \mapsto a_{i, j}(\lambda)$ are locally analytic -- avoiding the zeros of $g'$. But somehow, amazingly, $\gamma_{j}$ are co-analytic on $\C \setminus \R$.

When looking at \emph{bounded} Toeplitz operators on $H^2$, there is this result of Cowen \cite{Cowen} (see also \cite{Steph}).

\begin{Theorem}[Cowen]
Suppose that $\phi_1$ and $\phi_2$ are bounded rational functions on $\mathbb{D}$. Then the following are equivalent:
\begin{enumerate}
\item $T_{\phi_1}$ is similar to $T_{\phi_2}$.
\item $T_{\phi_1}$ is unitarily equivalent to $T_{\phi_2}$,
\item There is a bounded function $h$ on $\mathbb{D}$ and Blaschke products $b_1$ and $b_2$ of equal order such that $\phi_1 = h \circ b_1$ and $\phi_2 = h \circ b_2$.
\end{enumerate}
\end{Theorem}

Can we get a similar result for our unbounded Toeplitz operators? We think the answer is yes and we can prove the following result which is analogous to one direction of  Cowen's result \cite{Cowen} for bounded Toeplitz operators. In fact, with nearly the same proof.

\begin{Proposition}
If $g \in N^{+}_{\R}$ and $B$ is a finite Blaschke product of order $n$, then
$$T_{g \circ B} \cong \oplus_{n} T_{g}.$$
\end{Proposition}

\begin{proof}
Let $\{w_1, \ldots, w_n\}$ be an orthonormal basis for $(B H^2)^{\perp}$ (which is $n$-dimensional since $B$ has order $n$). Then $$\{w_{j} B^{k}: 1 \leqslant j \leqslant n, k \geqslant 1\}$$ is an orthonormal basis for $H^2$. This allows us to define the unitary operator
$$U: \oplus_{n} H^2 \to H^2, \quad U(\oplus_{j = 1}^{n} f_j) = w_1 (f_1 \circ B) + \cdots + w_{n} (f_n \circ B).$$
If $g = b/a$ is the canonical representation of $g$, as before, then, as discussed before, the domain of $T_g$ is $a H^2$, the domain of $T_{g \circ B}$ is $(a \circ B) H^2$, and the domain of $\oplus_{n} T_{g}$ is $\oplus_{n} aH^2$.

One easily checks from the definition of $U$ that $U(\oplus_{n} a H^2) = (a \circ B)H^2$ and that
$$U(\oplus_{n} T_{g}) = T_{g \circ B} U.$$ Thus $T_{g \circ B} \cong \oplus_{n} T_{g}.$
\end{proof}

\begin{Corollary}
If $g \in N^{+}_{\R}$ and $B_1, B_2$ are Blaschke products of order $n$, then $T_{g \circ B_1} \cong T_{g \circ B_2}$.
\end{Corollary}

\section{Herglotz spaces}
\label{section:Herglotz}

There are many ways one can create a model space $\mathcal{H}(\Gamma)$ for a given $T \in \mathcal{S}_{n}(\mathcal{H})$, i.e., a Hilbert space of vector-valued analytic functions on $\C \setminus \R$ for which multiplication by the independent variable is unitarily equivalent to $T$. Indeed, if $\mathcal{H}_1$ is a model space for $T$ and $W(z): \mathcal{K} \to \mathcal{K}$ is invertible for each $z \in \C \setminus \R$ and analytic on $\C \setminus \R$, then $\mathcal{H}_2 := W \mathcal{H}_1$ (endowed with the norm $\|W f\|_{\mathcal{H}_2} := \|f\|_{\mathcal{H}_1}$) is also a model space for $T$. That is to say the map $f \mapsto W f$ is an isometric multiplier from $\mathcal{H}_1$ onto $\mathcal{H}_2$. Furthermore, as seen by the proof of Corollary \ref{Cor-main-Liv-ue}, we know that if $K^1, K^2$ are the corresponding kernel functions for model spaces $\mathcal{H}_1, \mathcal{H}_2$ then
$$K^1_{\lambda}(z) = W(z) K^2_{\lambda}(z) W(\lambda)^{*}$$
if and only if $\mathcal{H}_1 = W \mathcal{H}_{2}$.

We know from our earlier work that, up to unitary operators (matrices), the Livsic function determines unitary equivalence for operators in $\mathcal{S}_{n}(\mathcal{H})$. It turns out that one can parameterize these model spaces in terms of the Livsic characteristic function and a certain Herglotz space. This will be the efforts of this section.

So far we know that for our given $T \in \mathcal{S}_{n}(\mathcal{H})$ and model $\Gamma$, the kernel function
$K_{\lambda}(z)$ can be factored as
$$K_{\lambda}(z) = \Phi(z) \left( \frac{I - V(z) V(\lambda)^{*}}{1 - b(z) \overline{b(\lambda)}} \right) \Phi(\lambda)^{*},$$
where $$\Phi(z) = K_{i}(z) K_{i}(i)^{-1/2}, \quad \Psi(z) = K_{-i}(z) K_{-i}(-i)^{-1/2}, \quad V(z) = b(z) \Phi(z)^{-1} \Psi(z),$$ and
$V$ is, up to unitary operators, the Livsic characteristic function for $T$. Moreover, $V$ is contractive on $\C_{+}$ and $V(i) = 0$. Also recall that $V$ is a meromorphic operator-valued function on $\C_{-}$.

As observed earlier in Remark \ref{zero-pole-V} but worth reminding here,  the denominator in the above formula for $K_{\lambda}(z)$ vanishes when $z = \overline{\lambda}$ and thus the numerator must also vanish. This shows
$$V(z) V(\overline{z})^{*} = I, \quad z \in \C \setminus \R.$$
When $V$ has a zero or pole at $z$ the above formula must be interpreted in the usual way (poles cancel out the zeros). This means that
we can use the the identity $V(z) V(\overline{z})^{*} = I$ along with the fact that $\|V(z)\| < 1$ for all $z \in \C_{+}$ and $\|V(z)\| > 1$ for all $z \in \C_{-}$, to see that
\begin{equation} \label{Omega}
\Omega(z) := (I + i V(z))(I - i V(z))^{-1}
\end{equation}
is well defined on $\C \setminus \R$. Moreover, one can check that
 \begin{enumerate}
 \item $z \mapsto \Omega(z)$ is an analytic operator-valued function on $\C \setminus \R$.
 \item $\Re \Omega(z)  := \frac{1}{2}(\Omega(z) + \Omega(z)^{*}) \geqslant 0$ on $\C_{+}$.
 \item $\Omega(z) = - \Omega(\ov{z})^{*}$.
 \end{enumerate}

 Such $\Omega$ satisfying the three properties listed above are called Herglotz functions and there is a very large theory of such functions \cite{MR0154132, DB, MR1784638, Langer}. The literature on this can be a bit confusing at times since Herglotz functions are often defined in slightly different ways or given different names, but they are essentially the same and have the same properties.

A computation will show the following.

\begin{Theorem}
If
$$W(z) := \sqrt{\pi}(z + i) \Phi(z) (\Omega(z) + I)^{-1}$$ then
$$K_{\lambda}(z) = W(z) \left(\frac{\Omega(z) + \Omega(\lambda)^{*}}{\pi i (\overline{\lambda}  - z)}\right) W(\lambda)^{*}.$$
\end{Theorem}

The function
$$K^{V}_{\lambda}(z) = \frac{\Omega(z) + \Omega(\lambda)^{*}}{\pi i (\overline{\lambda}  - z)}.$$
is a positive definite kernel function on $\C \setminus \R$ and, by general theory \cite{Paulsen}, is the reproducing kernel for a unique vector-valued reproducing kernel Hilbert space $\mathscr{H}(V)$, often called a \emph{Herglotz space}, and was discussed by L.~deBranges \cite{MR0154132, DB}. This gives us the following.

\begin{Theorem} \label{Mz-Herglotz}
Any $T \in \mathcal{S}_{n}$, $n \in \mathbb{N} \cup \{\infty\}$,  is unitarily equivalent to $M^{V}$, multiplication by the independent variable on a Herglotz space $\mathscr{H}(V)$, where $V$ is the Livsic function corresponding to $T$. Furthermore, the Livsic function for $M^{V}$ is  $V$.
\end{Theorem}

When $n < \infty$, we can use deBranges' results \cite{MR0154132, DB} further to identify the Herglotz space $\mathscr{H}(V)$ as a space of vector-valued Cauchy transforms. Indeed, by a vector-valued analog of the classical  Herglotz theorem (every positive harmonic function on $\C_{+}$ is the Poisson integral of a measure \cite{Duren}) there exists a positive matrix-valued measure $\mu$ on $\R$
satisfying (i) $\mu(E) \in M_{n \times n}(\C)$, $E \subset \R$, Borel; (ii) $\mu(E) \geqslant 0$ for all $E$; (iii) $\mu(\cup_j E_j) = \sum_{j} \mu(E_j)$, disjoint $E_j$; (iv)
$$\int \frac{d\langle \mu(t) a, a\rangle_{\C^n}}{1 + t^2} < \infty, \quad \forall a \in \C^n;$$
(v)
$$K^{V}_{\lambda}(z) = \frac{1}{\pi^2} \int \frac{d \mu(t)}{(t - \overline{\lambda})(t - z)},$$
i.e.,
$$\langle K^{V}_{\lambda}(z) a, b\rangle_{\C^n}  = \frac{1}{\pi^2} \int \frac{d\langle \mu(t) a, b\rangle}{(t - \overline{\lambda})(t - z)}, \quad \forall a, b \in \C^n;$$
(vi)
$$\mathscr{H}(V) = \left\{\frac{1}{\pi i} \int \frac{d \mu(t) f(t)}{t - z}: f \in L^2_{\C^n}(\mu)\right\}$$
and the Cauchy transform takes $L_{\C^n}^{2}(\mu) \to \mathscr{H}(V)$ in a unitary way.
What is vector-valued $L^{2}_{\C^n}(\mu)$? If
$$f = \sum_{j} c_j \chi_{E_{j}}, \quad c_{j} \in \C^n,$$ is a simple function,
where $\chi_{E_j}$ is a scalar-valued characteristic function on $\R$, define
$$\langle f, f \rangle_{L^{2}_{\C^n}(\mu)} := \sum_{j} \langle \mu(E_j) c_j, c_j\rangle_{\C^n}.$$
Now complete this to get an $L^2_{\C^n}(\mu)$ space. See \cite{MR0190760} for more on matrix and operator-valued measures.

In summary, we have the following:

\begin{Corollary}
Suppose $T \in \mathcal{S}_{n}(\mathcal{H})$, $n \in \mathbb{N}$. Then there is a positive matrix-valued measure $\mu$ on $\R$ satisfying the conditions above and such that $T$ is unitarily equivalent to $M_{\mu}$, multiplication by the independent variable with domain
$$\mathscr{D}(M_{\mu}) = \left\{f \in L^2_{\C^n}(\mu): x f \in L^2_{\C^n}(\mu), \int d\mu(t) f(t) = 0\right\}.$$
\end{Corollary}

\begin{Remark}
\begin{enumerate}
\item One can also prove this corollary by using the spectral theorem for a self-adjoint extension of $T$. See \cite{MR1821917} for details.
\item When $n = \infty$, identifying $\mathscr{H}(V)$ as a vector-valued $L^2$-type space becomes more difficult due to some convergence issues. However, in certain circumstances, e.g., when $\mu(E)$ is a trace-class operator for every Borel set $E$, one can identify $\mathscr{H}(V)$ as an $L^2$-type space. This is worked out carefully in \cite{MR0154132}.
\end{enumerate}
\end{Remark}

The function $\Omega$ in \eqref{Omega} can be replaced by
$$(I + A V(z))(I - A V(z))^{-1},$$
where $A \in U(n)$, the $n \times n$ unitary matrices, and an analogous result holds but with a positive $M_{n \times n}$-valued measure $\mu_{A}$. That is to say $T$ is unitarily equivalent to $M_{\mu_{A}}$, the densely defined multiplication by the independent variable on $L^2_{\C^n}(\mu_A)$. The family of measures $\{\mu_{A}: A \in U(n)\}$ is often called the family of \emph{Clark measures} \cite{CRM, Elliott, Martin-uni, Poltoratskii, Sarason-dB} corresponding to the function $V$ and have many fascinating properties. We will not go into the details here but one can show the following.

\begin{Theorem}
For $T_1 \in \mathcal{S}_{n}(\mathcal{H}_1), T_2 \in \mathcal{S}_n(\mathcal{H}_2)$ with corresponding Livsic functions $V_1, V_2$, we have that
$T_1 \cong T_2$ if and only if the associated family of Clark measures are the same.
\end{Theorem}

For a positive $M_{n \times n}$-valued measure $\mu$ satisfying the properties discussed above along with $\mu(\R) = \infty$, one can use Stieltjes inversion formula \cite{DB} to produce a $V$ in the closed unit ball of $H^{\infty}_{\C^n}(\C_{+})$ ($\C^n$-valued bounded analytic functions on $\C_{+}$) such that $\mu$ belongs to the Clark family of measures corresponding to $V$. Moreover $V$ will be the Livsic function corresponding to $M_{\mu}$. This tells is the following:

\begin{Corollary}
For positive $M_{n \times n}$-valued measures $\mu, \nu$ above we have that $M_{\mu} \cong M_{\nu}$ if and only if $\mu, \nu$ belong to the same Clark family corresponding to some $V$ in the unit ball of $H^{\infty}_{\C^n}(\C_{+})$.
\end{Corollary}

We will point out that determining when $\mu, \nu$ belong to the same Clark family seems to be a difficult problem.


\section{deBranges-Rovnyak spaces}
\label{section:dBR}

In this final section, we will show that when $V$, the Livsic function for $T \in \mathcal{S}_{n}(\mathcal{H})$, $n < \infty$,  is an extreme function for $\mathscr{B}_{\C^n}$, the closed unit ball in $H^{\infty}_{\C^n}(\C_{+})$,  then $M^{V}$ on the  Herglotz space $\mathscr{H}(V)$ is unitarily equivalent to multiplication by the independent variable on a vector-valued deBrange-Rovnyak space. In the examples we covered, differentiation operators, Sturm-Liouville operators, Toeplitz operators, etc., we will, through the Livsic functions we computed earlier, connect these operators to multiplication operators on these deBranges-Rovnyak spaces. Along the way, we will show an interesting property of the Livsic function.

Compare the formula \eqref{dbrkernel} for the reproducing kernels of the de Branges-Rovnyak space $\mathscr{K}(V)$ with the formulas for the reproducing kernels of the representation space $\mathcal{H} (\Gamma)$ as given in Theorem \ref{T-main-vector-kernel},
\[ K_w (z) = \Phi (z)  \left( \frac{ I - V(z) V(w ) ^* }{1 - \overline{b(w)}b(z)} \right) \Phi (w) ^* \] for any $w, z \in \mathbb{C} \setminus \mathbb{R}$, where
\[ \Phi (z) = K_i (z) K_i (i) ^{-1/2 },  \] and $K_i (z) = \Gamma (z) ^* \Gamma ( i)$.  Now let
$$\mathcal{H} (\Gamma) _+ := \bigvee _{\lambda \in \mathbb{C} _+ } K _\lambda  \C^n \subset \mathcal{H} (\Gamma ).$$ Similarly let
$$\mathscr{H} (V) _+  := \bigvee_{\lambda \in \C_{+}}  K_{\lambda}^{V} \C^n \subset \mathscr{H}(V).$$
  It follows from Section \ref{section:Herglotz} that multiplication by $$W(z) := \sqrt{\pi} (z+i) \Phi (z) ( \Omega(z) + I ) ^{-1}$$ is an isometry of $\mathscr{H} (V)$ onto $\mathcal{H} (\Gamma )$ which takes $\mathscr{H} (V) _+$ onto $\mathcal{H}  (\Gamma) _+$. This next theorem shows that there is also a natural isometric multiplier from $\mathscr{H}  (V) _+$ onto $\mathscr{K}(V)$.

\begin{Theorem}
    Multiplication by $ U(z) = \sqrt{\pi} (z+i) \Phi (z)$ is an isometry of $\mathscr{K}(V)$ onto $\mathcal{H}  (\Gamma) _+$, and hence $Q := \frac{1}{2} (I - V)$ is an isometric multiplier of $\mathscr{H} (V ) _+ $  onto $\mathscr{K}(V)$. \label{multiplier}
\end{Theorem}

\begin{proof}
    Since $$ (I + \Omega) ^{-1} = \frac{I +V}{2} $$ we see that if we can show that $U$ is an isometric multiplier of $\mathscr{K}(V)$ onto $\mathcal{H}  (\Gamma ) _+$, then since
    $$W(z) = \sqrt{\pi} (z+i) \Phi (z) (\Omega(z) + I) ^{-1}$$
is an isometric multiplier of $\mathscr{H} (V) _+ $ onto $\mathcal{H} (\Gamma ) _+,$ it will follow that $$W U^{-1} = \frac{I +V}{2} = Q $$ is an isometric multiplier of
$\mathscr{H} (V) _+$ onto $\mathscr{K}(V)$.

To see that $U$ is an isometric multiplier from $\mathscr{K}(V)$ onto $\mathcal{H} (\Gamma ) _+$, it suffices to verify, as discussed in Section \ref{section:Herglotz}, that $$K _\lambda (z) = U(z) \Delta ^V _\lambda (z) U (\lambda) ^*,$$
where $\Delta _\lambda ^V $ and $K_\lambda$ are as above. It is indeed easy to check that
\begin{eqnarray}
K_\lambda (z) & = & \Phi (z) \left( \frac{ I - V(z) V(\lambda) ^*}{I-\overline{b(\lambda)} b(z)}  \right) \Phi (\lambda) ^*  \nonumber \\
& = & \sqrt{\pi} (z+i) \Phi (z) \left( \frac{i}{2\pi} \frac{ I - V(z) V(\lambda) ^*}{z -\overline{\lambda}} \right) \left( \sqrt{\pi} (\lambda +i ) \Phi (\lambda)  \right) ^* \nonumber \\
& =& U(z) \Delta _\lambda ^V (z) U(\lambda) ^* .\end{eqnarray}
This proves the claim.
\end{proof}

It can be shown  \cite{Martin-uni} that if $V$ (the Livsic characteristic function) is an extreme point of $\mathscr{B}_{\C^n}$ that
$$\mathscr{H} (V) _+ = \mathscr{H} (V)$$ so that $Q$ is an isometric multiplier from $\mathscr{H} (V) $ onto the de Branges-Rovnyak space $\mathscr{K}(V)$.
More precisely, as was discussed in \cite[Section 4.3]{Martin-uni} the Helson-Lowdenslager generalization of Szego's theorem \cite[Theorem 8]{HelsonIII} allows us to characterize the extreme points of
$\mathscr{B}_{\C^n}$ as follows.

\begin{Theorem}
Given $V \in \mathscr{B}_{\C^n}$, the following are equivalent:
\begin{enumerate}
    \item $V$ is an extreme point.
    \item $$ \int _{-\infty} ^\infty \mathrm{tr} \left( \log (I - | V (x) | ) \right) \frac{1}{1+x^2} dx = - \infty.$$
    \item $\mathscr{H} (V) _+ = \mathscr{H} (V)$.

\end{enumerate}
\end{Theorem}

\begin{Remark}
\begin{enumerate}
\item  The theorem above is actually a translation of the results of \cite[Section 4.3]{Martin-uni}, which were originally stated for contractive matrix analytic functions on the unit disc,
to the setting of the upper half-plane.
\item It follows that if $n < \infty $ and $V$ is an extreme point, that $Q  = \frac{1}{2} (I - V)$ is an isometric multiplier of the Herglotz space $\mathcal{H}  (V ) = \mathcal{H} (V) _+ $  onto
the de Branges-Rovnyak space $\mathscr{K}(V)$. While this fact may still hold in the case where $n=\infty$, our only known proof of the implication $(1) \Rightarrow (3)$ in the above
theorem uses the condition $(2)$, and it is not clear how to formulate $(2)$ in the case where $n=\infty$. Moreover the proof that $(2) \Rightarrow (3)$ uses the Helson-Lowdenslager
generalization of Szego's theorem, and it is not immediately clear whether there is an analogue of this theorem in the case where $n=\infty$, or whether there is a way to directly
prove the implications $(1) \Leftrightarrow (3)$.
\end{enumerate}
\end{Remark}

There is a nice corollary to this result along with Theorem \ref{Martin-Z} which applies, in particular, to the operators mentioned throughout this paper: differentiation, double differentiation, Sturm-Liouville, Toeplitz, etc. For these operators we have the following.

\begin{Corollary} \label{Mz-DR}
If $T \in \mathcal{S}_{n}(\mathcal{H}), n < \infty$, and its Livsic characteristic function $V$ is an extreme point of $\mathscr{B}_{\C^n}$, then $T$ is unitarily equivalent to $Z_{V}$, multiplication by the independent variable in the deBranges-Rovnyak space $\mathscr{K}(V)$. Furthermore, $(V \circ b^{-1}) \vec{k}$ does not have an angular derivative at $z = 1$ for any $\vec{k} \in \C^n$.
\end{Corollary}


\begin{thebibliography}{10}

\bibitem{MR0167642}
M.~Abramowitz and I.~A. Stegun, \emph{Handbook of mathematical functions with
  formulas, graphs, and mathematical tables}, National Bureau of Standards
  Applied Mathematics Series, vol.~55, For sale by the Superintendent of
  Documents, U.S. Government Printing Office, Washington, D.C., 1964.

\bibitem{AC}
P.~R. Ahern and D.~N. Clark, \emph{Radial limits and invariant subspaces},
  Amer. J. Math. \textbf{92} (1970), 332--342.

\bibitem{A-G}
N.~I. Akhiezer and I.~M. Glazman, \emph{Theory of linear operators in {H}ilbert
  space. {V}ol. {II}}, Translated from the Russian by Merlynd Nestell,
  Frederick Ungar Publishing Co., New York, 1963.

\bibitem{A-F-R}
A.~Aleman, N.~Feldman, and W.~T. Ross, \emph{The {H}ardy space of a slit
  domain}, Frontiers in Mathematics, Birkhauser, Basel, 2009.

\bibitem{MR0190760}
S.~K. Berberian, \emph{Notes on spectral theory}, Van Nostrand Mathematical
  Studies, No. 5, D. Van Nostrand Co., Inc., Princeton, N.J.-Toronto,
  Ont.-London, 1966.

\bibitem{CR}
J.~A. Cima and W.~T. Ross, \emph{The backward shift on the {H}ardy space},
  Mathematical Surveys and Monographs, vol.~79, American Mathematical Society,
  Providence, RI, 2000.

\bibitem{CRM}
Joseph~A. Cima, Alec~L. Matheson, and William~T. Ross, \emph{The {C}auchy
  transform}, Mathematical Surveys and Monographs, vol. 125, American
  Mathematical Society, Providence, RI, 2006.

\bibitem{Cowen}
C.~Cowen, \emph{On equivalence of {T}oeplitz operators}, J. Operator Theory
  \textbf{7} (1982), no.~1, 167--172.

\bibitem{Cow-Doug}
M.~J. Cowen and R.~G. Douglas, \emph{Complex geometry and operator theory},
  Acta Math. \textbf{141} (1978), no.~3-4, 187--261.

\bibitem{MR0154132}
L.~de~Branges, \emph{Perturbations of self-adjoint transformations}, Amer. J.
  Math. \textbf{84} (1962), 543--560.

\bibitem{DB}
\bysame, \emph{Hilbert spaces of entire functions}, Prentice-Hall Inc.,
  Englewood Cliffs, N.J., 1968.

\bibitem{Duren}
P.~L. Duren, \emph{Theory of ${H}\sp{p}$ spaces}, Academic Press, New York,
  1970.

\bibitem{Elliott}
S.~Elliott, \emph{A matrix-valued {A}leksandrov disintegration theorem},
  Complex Anal. Oper. Theory \textbf{4} (2010), no.~2, 145--157.

\bibitem{Garnett}
John~B. Garnett, \emph{Bounded analytic functions}, first ed., Graduate Texts
  in Mathematics, vol. 236, Springer, New York, 2007. \MR{2261424
  (2007e:30049)}

\bibitem{MR1821917}
F.~Gesztesy, N.~Kalton, K.~A. Makarov, and E.~Tsekanovskii, \emph{Some
  applications of operator-valued {H}erglotz functions}, Operator theory,
  system theory and related topics ({B}eer-{S}heva/{R}ehovot, 1997), Oper.
  Theory Adv. Appl., vol. 123, Birkh\"auser, Basel, 2001, pp.~271--321.

\bibitem{MR1784638}
F.~Gesztesy and E.~Tsekanovskii, \emph{On matrix-valued {H}erglotz functions},
  Math. Nachr. \textbf{218} (2000), 61--138.

\bibitem{Gil1}
R.~C. Gilbert, \emph{Simplicity of linear ordinary differential operators}, J.
  Differ. Equations \textbf{11} (1972), 672--681.

\bibitem{Gil2}
\bysame, \emph{Simplicity of differential operators on an infinite interval},
  J. Differ. Equations \textbf{14} (1973), 1--8.

\bibitem{MR1466698}
M.~L. Gorbachuk and V.~I. Gorbachuk, \emph{M. {G}. {K}rein's lectures on entire
  operators}, Operator Theory: Advances and Applications, vol.~97, Birkh\"auser
  Verlag, Basel, 1997.

\bibitem{Helson}
H.~Helson, \emph{Large analytic functions}, Linear operators in function spaces
  ({T}imi\c soara, 1988), Oper. Theory Adv. Appl., vol.~43, Birkh\"auser,
  Basel, 1990, pp.~209--216.

\bibitem{Helson2}
\bysame, \emph{Large analytic functions. {II}}, Analysis and partial
  differential equations, Lecture Notes in Pure and Appl. Math., vol. 122,
  Dekker, New York, 1990, pp.~217--220.

\bibitem{HelsonIII}
H.~Helson and D.~Lowdenslager, \emph{Prediction theory and {F}ourier series in
  several variables},  (1958).

\bibitem{Hille}
E.~Hille, \emph{Ordinary differential equations in the complex domain},
  American Mathematical Society, Providence, RI, 1976.

\bibitem{MR0011170}
M.~Krein, \emph{On {H}ermitian operators whose deficiency indices are {$1$}},
  C. R. (Doklady) Acad. Sci. URSS (N.S.) \textbf{43} (1944), 323--326.

\bibitem{MR0011533}
\bysame, \emph{On {H}ermitian operators with deficiency indices equal to one.
  {II}}, C. R. (Doklady) Acad. Sci. URSS (N. S.) \textbf{44} (1944), 131--134.

\bibitem{MR0048704}
M.~G. Kre{\u\i}n, \emph{The fundamental propositions of the theory of
  representations of {H}ermitian operators with deficiency index {$(m,m)$}},
  Ukrain. Mat. \v Zurnal \textbf{1} (1949), no.~2, 3--66.

\bibitem{Langer}
H.~Langer and B.~Textorius, \emph{On generalized resolvents and {$Q$}-functions
  of symmetric linear relations (subspaces) in {H}ilbert space}, Pacific J.
  Math. \textbf{72} (1977), no.~1, 135--165.

\bibitem{Livsic-1}
M.~S. Liv{\v{s}}ic, \emph{Isometric operators with equal deficiency indices,
  quasi-unitary operators}, Amer. Math. Soc. Transl. (2) \textbf{13} (1960),
  85--103.

\bibitem{Livsic-2}
\bysame, \emph{On a class of linear operators in {H}ilbert space}, Amer. Math.
  Soc. Transl. (2) \textbf{13} (1960), 61--83.

\bibitem{MR2215727}
N.~Makarov and A.~Poltoratski, \emph{Meromorphic inner functions, {T}oeplitz
  kernels and the uncertainty principle}, Perspectives in analysis, Math. Phys.
  Stud., vol.~27, Springer, Berlin, 2005, pp.~185--252.

\bibitem{Martin1}
R.~T.~W. Martin, \emph{Representation of simple symmetric operators with
  deficiency indices {$(1,1)$} in de {B}ranges space}, Complex Anal. Oper.
  Theory \textbf{5} (2011), no.~2, 545--577.

\bibitem{Martin-uni}
\bysame, \emph{Unitary perturbations of compressed n-dimensional shifts},
  Compl. Anal. Oper. Theory. \ In press: http://arxiv.org/abs/1107.3439 (2012).

\bibitem{Naimark}
M.A. Naimark, \emph{Linear differential operators {V}ol. {II}}, Frederick Ungar
  Publishing Co., New York, 1969.

\bibitem{Nik}
N.~K. Nikolski, \emph{Operators, functions, and systems: an easy reading.
  {V}ol. 1}, Mathematical Surveys and Monographs, vol.~92, American
  Mathematical Society, Providence, RI, 2002, Hardy, Hankel, and Toeplitz,
  Translated from the French by Andreas Hartmann. \MR{1864396 (2003i:47001a)}

\bibitem{Nik2}
\bysame, \emph{Operators, functions, and systems: an easy reading. {V}ol. 2},
  Mathematical Surveys and Monographs, vol.~93, American Mathematical Society,
  Providence, RI, 2002, Model operators and systems, Translated from the French
  by Andreas Hartmann and revised by the author.

\bibitem{Niktr}
\bysame, \emph{Treatise on the shift operator},
  Grundlehren der Mathematischen Wissenschaften [Fundamental Principles of
  Mathematical Sciences], vol. 273, Springer-Verlag, Berlin, 1986, Spectral
  function theory, With an appendix by S. V. Hru{\v{s}}{\v{c}}ev [S. V.
  Khrushch{\"e}v] and V. V. Peller, Translated from the Russian by Jaak Peetre.

\bibitem{Paulsen}
V.~Paulsen, \emph{An introduction to the theory of reproducing kernel {H}ilbert
  spaces}, www.math.uh.edu/{~}vern/rkhs.pdf, 2009.

\bibitem{Poltoratskii}
A.~Poltoratski, \emph{Boundary behavior of pseudocontinuable functions},
  Algebra i Analiz \textbf{5} (1993), no.~2, 189--210.

\bibitem{Remling}
C.~Remling, \emph{Schrodinger operators and de {B}ranges spaces}, J. Funct.
  Anal. \textbf{196} (2002), 323--394.

\bibitem{RS}
W.~T. Ross and H.~S. Shapiro, \emph{Generalized analytic continuation},
  University Lecture Series, vol.~25, American Mathematical Society,
  Providence, RI, 2002.

\bibitem{Sarason-dB}
D.~Sarason, \emph{Sub-{H}ardy {H}ilbert spaces in the unit disk}, John Wiley \&
  Sons Inc., New York, NY, 1994.

\bibitem{Sarason}
D.~Sarason, \emph{Unbounded {T}oeplitz operators}, Integral Equations Operator
  Theory \textbf{61} (2008), no.~2, 281--298.

\bibitem{Steph}
K.~Stephenson, \emph{Analytic functions of finite valence, with applications to
  {T}oeplitz operators}, Michigan Math. J. \textbf{32} (1985), no.~1, 5--19.

\bibitem{MR1349110}
G.~N. Watson, \emph{A treatise on the theory of {B}essel functions}, Cambridge
  Mathematical Library, Cambridge University Press, Cambridge, 1995, Reprint of
  the second (1944) edition.

\end{thebibliography}
\end{document}